\documentclass[12pt,twoside]{amsart}
\usepackage[all]{xy}
\usepackage{graphicx}
\usepackage{amsmath}

\usepackage{longtable}

\title[a toric weak Fano 4-fold to be deformed to Fano]
{A sufficient condition for a toric weak Fano 4-fold\\ 
to be deformed to a Fano manifold}
\author{Hiroshi Sato} 
\date{2021/4/22, version 0.26}
\address{Department of Applied Mathematics, Faculty of Sciences, 
Fukuoka University, 8-19-1, Nanakuma, Jonan-ku, Fukuoka 814-0180, Japan}
\email{hirosato@fukuoka-u.ac.jp}

\makeatletter
    
    \@addtoreset{equation}{section}
\makeatother

\newcommand{\Pic}[0]{{\operatorname{Pic}}}

\newcommand{\G}[0]{{\operatorname{G}}}

\newcommand{\Hom}[0]{{\operatorname{Hom}}}

\newtheorem{thm}{Theorem}[section]
\newtheorem{lem}[thm]{Lemma}

\newtheorem{prop}[thm]{Proposition}
\newtheorem*{claim}{Claim}

\theoremstyle{definition}
\newtheorem{ex}[thm]{Example}
\newtheorem{defn}[thm]{Definition}

\newtheorem{rem}[thm]{Remark}
\newtheorem*{ack}{Acknowledgments}       

\begin{document}
\bibliographystyle{amsalpha+}

\begin{abstract}
In this paper, we introduce the notion of 
toric special weak Fano manifolds, 
which have only special primitive crepant contractions. 
We study their structure, 
and in particular completely classify smooth toric special 
weak Fano $4$-folds. 
As a result, we can confirm that almost every smooth toric special weak Fano 
$4$-fold is a weakened Fano manifold, that is, 
a weak Fano manifold which can be deformed to a Fano manifold. 
\end{abstract}

\thispagestyle{empty}

\maketitle

\tableofcontents
\section{Introduction} 
The deformation theory of Hirzebruch surfaces are well-known. The deformation 
type of the Hirzebruch surface $F_a$ is determined by the parity of its degree $a$. 
In particular, the weak del Pezzo surface $F_2$ 
is deformed to the del Pezzo surface $F_0\cong\mathbb{P}^1\times\mathbb{P}^1$, 
and this phenomenon is very interesting. As a generalization of this result, 
Minagawa studied the deformations of primitive crepant contractions 
on a smooth weak Fano $3$-fold in \cite{minagawa}, and introduced 
the following notion of weakened Fano manifolds in \cite{minagawaweakened}:

\begin{defn}\label{weakeneddef}
Let $X$ be a smooth weak Fano $d$-fold over 
the complex number field $\mathbb{C}$. Then, $X$ is 
a {\em weakened} Fano manifold if the following hold:
\begin{enumerate}
\item 
$X$ is not a Fano manifold.
\item 
There exists a one-parameter deformation 
\[
\mathfrak{f}:\mathcal{X}\to\Delta_\epsilon:=
\left\{c\in\mathbb{C}\,\left|\,|c|<\epsilon\right.\right\}
\]
for sufficiently small $\epsilon>0$ such that the central fiber 
$\mathcal{X}_0:=\mathfrak{f}^{-1}(0)$ is 
isomorphic to $X$, while any general fiber 
$\mathcal{X}_t:=\mathfrak{f}^{-1}(t)$ is a Fano manifold 
for any $t\in\Delta_\epsilon\setminus\{0\}$. 
\end{enumerate}
\end{defn}

For $d=2$, every smooth weak del Pezzo surface is always a weakened del Pezzo 
surface if it is not a del Pezzo surface. 
Smooth toric weakened Fano $3$-folds were classified in \cite{sato2} by 
using the characterization of smooth weakened Fano $3$-folds in 
\cite{minagawa}, which says that 
a smooth weak Fano $3$-fold $X$ is a weakened Fano manifold 
if and only if every primitive crepant contraction on $X$ is of $(0,2)$-{\em type}, 
and $X$ is not a Fano manifold. 
In this paper, based on this characterization theorem and 
the deformation theory of toric manifolds in \cite{laface}, 
we define smooth toric {\em special} weak Fano $d$-folds 
(see Definitions \ref{specon} and \ref{mizukawa}). 
These varieties have common geometric 
properties (see Propositions \ref{3nenagumi} and \ref{ichiomain}). As an application, we 
classify all the smooth toric special weak Fano $4$-folds in Section \ref{sec4fold} 
as follows: 
\begin{thm}[Theorem \ref{4foldthm}]
There exist exactly $26$ smooth toric special weak Fano $4$-folds 
which do not decompose into a direct product of 
lower-dimensional varieties. 
\end{thm}
Moreover, by using the deformation families 
of toric manifolds constructed in \cite{laface}, we can see 
almost every smooth toric special weak Fano $4$-fold is deformation 
equivalent to a Fano manifold as follows: 
\begin{thm}[Theorem \ref{deformmain}]
Let $X$ be a smooth toric special weak Fano $4$-fold. 
In all but five cases, 
$X$ is deformation equivalent to a toric Fano manifold. 
In particular, $X$ is a 
smooth toric weakened Fano $4$-fold with the exception of 
two unknown cases. 
\end{thm}
Thus, the condition to be smooth toric 
special weak Fano $4$-folds is also a sufficient condition for 
smooth toric weakened Fano $4$-folds under some assumption. 
This result seems to be the first step toward the study of 
the deformation theory of primitive crepant contractions 
on higher-dimensional smooth weak Fano varieties. 

This paper is organized as follows: 
In Section \ref{prepre}, we collect some basic results of toric geometry 
such as primitive collections and relations, toric Mori theory, 
toric bundles and so on. In Section \ref{specialfanosec}, 
we introduce the concept of smooth toric special 
weak Fano varieties, which are the main objects in this paper, and 
study their structure. In Section \ref{sec4fold}, 
we completely classify all the smooth toric special 
weak Fano $4$-folds, and investigate their deformations 
in Section \ref{deformsec}.

\begin{ack}
The author was partially supported by JSPS KAKENHI 
Grant Number JP18K03262. 
He would like to thank Professor Hirokazu Nasu, who 
kindly answered his questions about deformation theory. 
He would also like to thank Professors Osamu Fujino and 
Yukari Ito, who gave him useful comments 
about the descriptions of Proposition \ref{thanksfujino} and 
Lemma \ref{anticanonical}. 
Finally, 
he thanks the referee very much 
for many useful comments and suggestions. 
\end{ack}

\section{Preliminaries}\label{prepre}
In this section, we prepare basic results in the toric geometry and 
the toric Mori theory. For the details, please see 
\cite{cls}, \cite{fulton} and \cite{oda} for toric geometry, and 
see \cite{fujino-sato}, \cite{matsuki} and \cite{reid} for 
toric Mori theory. Also refer \cite{bat1}, \cite{bat2}, 
\cite{casagrande} and \cite{sato1} for the concepts of 
primitive collections and relations, which are very useful for our theory. 
We will work over the complex number field $\mathbb{C}$ throughout 
this paper. 

Let $N:=\mathbb{Z}^d$, $M:=\Hom_{\mathbb{Z}}(N,\mathbb{Z})$, 
$N_{\mathbb{R}}:=N\otimes_{\mathbb{Z}}\mathbb{R}$ and 
$M_{\mathbb{R}}:=M\otimes_{\mathbb{Z}}\mathbb{R}$. 
$\langle n_1,\ldots,n_r\rangle$ 
stands for the cone generated by $n_1,\ldots,n_r\in N_{\mathbb{R}}$. 
For any two cones $\sigma,\sigma'\subset N_{\mathbb{R}}$, 
put 
\[
\sigma+\sigma':=\left\{n+n'\in N_{\mathbb{R}}\,\left|\,n\in\sigma
\mbox{ and }
n'\in\sigma'\right.\right\}.
\]
For a fan $\Sigma$ in $N$, we denote by $X=X_\Sigma$ 
the toric $d$-fold associated to $\Sigma$. 
For a smooth complete fan $\Sigma$, put 
\[
\G(\Sigma):=\left\{\mbox{the primitive generators of 
$1$-dimensional cones in $\Sigma$}\right\}\subset N.
\]
For any $x\in\G(\Sigma)$, there is the corresponding torus invariant 
divisor $D_x$ on $X$. We have the following fundamental 
exact sequence
\[
0\to M\to \mathbb{Z}^{\G(\Sigma)}\to \Pic (X)\to 0,
\]
where we regard $\mathbb{Z}^{\G(\Sigma)}$ as the 
group $\sum_{x\in\G(\Sigma)}\mathbb{Z}D_x$ 
of torus invariant divisors on $X$, naturally. In its dual exact sequence
\[
0\gets N\xleftarrow{\psi} \left(\mathbb{Z}^{\G(\Sigma)}
\right)^*\gets
{\rm A}^1(X)\gets 0,
\]
where ${\rm A}^1(X)$ is the group of numerical $1$-cycles on $X$, 
we have 
\[
\psi\left((a_x)_{x\in\G(\Sigma)}\right)=
\sum_{x\in\G(\Sigma)}a_xx
\]
for $(a_x)_{x\in\G(\Sigma)}\in \left(\mathbb{Z}^{\G(\Sigma)}
\right)^*$. 
Therefore, we obtain the following isomorphism of abelian groups:
\[
{\rm A}^1(X)\cong 
\left\{(a_x)_{x\in\G(\Sigma)}\in \left(\mathbb{Z}^{\G(\Sigma)}\right)^*\,\left|\,
\sum_{x\in\G(\Sigma)}a_xx=0\right.\right\}\subset \left(\mathbb{Z}^{\G(\Sigma)}\right)^*.
\]
Thus, we can regard a linear relation among elements in $\G(\Sigma)$ as 
a numerical $1$-cycle on $X$. In particular, the following relation is important for our theory:
\begin{defn}[{\cite[Definitions 2.6, 2.7 and 2.8]{bat1}}]\label{pcpc}
Let $X=X_\Sigma$ be a smooth complete toric variety. 
We call a nonempty subset $P\subset \G(\Sigma)$ a {\em primitive collection} in 
$\Sigma$ if $P$ does not generate a cone in $\Sigma$, while 
any proper subset of $P$ generates a cone in $\Sigma$. 

For a primitive collection $P=\{x_1,\ldots,x_r\}$, there exists the unique cone 
$\sigma(P)\in\Sigma$ such that $x_1+\cdots+x_r$ is contained in the 
relative interior of $\sigma(P)$. Put $\sigma(P)\cap\G(\Sigma)=
\{y_1,\ldots,y_s\}$. Then, we have a linear relation
\[
x_1+\cdots+x_r=a_1y_1+\cdots+a_sy_s\ (a_1,\ldots,a_s\in\mathbb{Z}_{>0}).
\]
We call this relation the {\em primitive relation} of $P$. 
Thus, by the above argument, 
we obtain the numerical $1$-cycle $r(P)\in{\rm A}^1(X)$ for any 
primitive collection $P\subset\G(\Sigma)$. 
We put $\deg P:=r-(a_1+\cdots+a_s)$. 
We remark that $\deg P$ is 
the intersection number $\left(-K_X\cdot r(P)
\right)$.
\end{defn}
Primitive collections and relations are very important in toric Mori theory as follows:

\begin{thm}[{\cite[Theorem 2.15]{bat1}}]\label{toricconethm}
Let $X=X_\Sigma$ be a smooth projective toric variety. 
Then, the Kleiman-Mori cone ${\rm NE}(X)$ of $X$ is 
generated by the $1$-cycles associated to primitive relations. Namely,
\[
{\rm NE}(X)=\sum_{P:\mbox{ primitive collection in }\Sigma} 
{\mathbb{R}_{\ge 0}}r(P)\subset {\rm N}^1(X):={\rm A}^1(X)\otimes \mathbb{R}.
\]
When $r(P)$ is contained in an extremal ray of ${\rm NE}(X)$, 
we call $P$ an {\em extremal} primitive collection. 
\end{thm}

\begin{rem}\label{odasense}
It is well-known that for a smooth projective toric variety $X=X_\Sigma$, 
${\rm NE}(X)$ is generated by torus invariant curves on $X$. 
More strongly, for any curve $C\subset X$, there exist 
torus invariant curves $C_1,\ldots,C_l$ and positive {\em integers} 
$b_1,\ldots,b_l\in\mathbb{Z}_{> 0}$ such that 
\[
C=b_1C_1+\cdots+b_lC_l
\]
in ${\rm N}^1(X)$. The proof is contained in the Japanese version of 
\cite{oda} (also see \cite[Proposition 2.3]{sasuke}). 
\end{rem}

For an {\em extremal} primitive collection 
$P\subset\G(\Sigma)$, the associated extremal 
contraction can be described by using 
its primitive relation. Let 
$x_1+\cdots+x_r=a_1y_1+\cdots+a_sy_s$ 
be its primitive relation, where $G:=\{x_1,\ldots,x_r,y_1,\ldots,
y_s\}\subset\G(\Sigma)$ and $a_1,\ldots,a_s
\in\mathbb{Z}_{> 0}$, 
and $\varphi_R:X\to\overline{X}$ 
the extremal contraction associated to the extremal ray 
\[
R=\mathbb{R}_{\ge 0}r(P)\subset{\rm NE}(X).
\]
The following is the key to constructing $\varphi_R$. 
Also, this proposition is essential for our theory in this paper: 
\begin{prop}[{\cite[Theorem 2.2]{casagrande}}]\label{contractible}   
For any cone $\tau
\in\Sigma$ such that $\left(\tau\cap\G(\Sigma)
\right)
\cap
G=
\emptyset$ and $\langle y_1,\ldots,y_s\rangle
+\tau\in\Sigma$, we have 
\[
\langle x_1,\ldots,\stackrel{\vee}{x_{i}},
\ldots,x_r,y_1,\ldots,y_s
\rangle +\tau\in\Sigma
\]
for any $1\le i\le r$. 
\end{prop}
$\varphi_R$ is constructed as follows. 
We divide into three cases:
\begin{enumerate}
\item $\varphi_R$ is a Fano contraction 
$\Longleftrightarrow$ $s=0$. Let $\mathcal{N}\subset N_{\mathbb{R}}$ 
be the subspace spanned by $\{x_1,\ldots,x_r\}$. Then, 
$\varphi_R$ is induced by the surjective morphism 
$N_{\mathbb{R}}\to N_{\mathbb{R}}/\mathcal{N}$. 
$X$ has a $\mathbb{P}^{r-1}$-bundle structure over $\overline{X}$. 
\item $\varphi_R$ is a divisorial contraction 
$\Longleftrightarrow$ $s=1$. Let 
\[
\overline{\Sigma}:=\left\{\sigma\in\Sigma\,\left|\,
y_1\not\in \sigma\cap\G(\Sigma)
\right.\right\}\cup
\]
\[
\left\{\,
\langle x_1,\ldots,x_r
\rangle +\tau
\left|\,
\tau
\in\Sigma,\ \left(\tau\cap\G(\Sigma)
\right)
\cap G=
\emptyset\mbox{ and }
\langle y_1\rangle
+\tau\in\Sigma
\right.\right\}.
\]
Then, $\overline{\Sigma}$ is the fan associated to $\overline{X}$. 
$\varphi_R$ contracts the torus invariant divisor corresponding to 
the cone $\langle y_1\rangle\in\Sigma$ to 
the $(d-r)$-dimensional torus invariant subvariety corresponding to 
the cone $\langle x_1,\ldots,x_r \rangle\in\overline{\Sigma}$. 
\item $\varphi_R$ is a small contraction 
$\Longleftrightarrow$ $s\ge 2$. 
Let 
\[
\overline{\Sigma}:=\left\{\sigma\in\Sigma\,\left|\,
\{ y_1,\ldots,y_s\}\not\subset\sigma\cap\G(\Sigma)
\right.\right\}\cup
\]
\[
\left\{\,
\langle x_1,\ldots,x_r,y_1,\ldots,y_s
\rangle +\tau
\left|\,
\tau
\in\Sigma,\ \left(\tau\cap\G(\Sigma)
\right)
\cap G=
\emptyset\mbox{ and }
\langle y_1,\ldots,y_s\rangle
+\tau\in\Sigma
\right.\right\}.
\]
Then, $\overline{\Sigma}$ is the fan associated to $\overline{X}$. 
$\varphi_R$ contracts the $(d-s)$-dimensional 
torus invariant subvariety corresponding to 
the cone $\langle y_1,\ldots,y_s\rangle\in\Sigma$ to 
the $(d-r-s+1)$-dimensional torus invariant subvariety corresponding to 
the cone 
\[
\langle x_1,\ldots,x_r,y_1,\ldots,y_s \rangle
\in\overline{\Sigma}. 
\]
$\overline{X}$ has a non-$\mathbb{Q}$-factorial singularity. 
\end{enumerate}
Moreover, the following hold: 
$R$ is $K_X$-negative 
if and only if $\deg P>0$, while 
$\varphi_R$ is a primitive {\em crepant} 
contraction 
if and only if $\deg P=0$. 

\medskip 

The following is a characterization of 
a toric bundle structure using primitive 
collections: 
Let $X=X_{\Sigma}$ be a smooth projective toric $d$-fold. 
Suppose that there exist nonempty subsets
$G_1\subset \G(\Sigma)$ and $G_2\subset \G(\Sigma)$ 
such that $G_1\cap G_2=\emptyset$, $G_1\cup G_2=\G(\Sigma)$ and 
for any primitive collection $P\subset \G(\Sigma)$, 
either $P\subset G_1$ or $P\subset G_2$ holds. 
This condition means that $\Sigma$ is equivalent to the fan of 
a direct product as a simplicial complex 
(see e.g. \cite[Proposition 3.1.14]{cls}). 
Put 
\[
\Sigma_1:=\left\{\sigma\in\Sigma\,\left|\,\sigma\cap\G(\Sigma)\subset G_1\right.\right\}
\mbox{ and }
\Sigma_2:=\left\{\sigma\in\Sigma\,\left|\,\sigma\cap\G(\Sigma)\subset G_2\right.\right\}. 
\]
Then, for some $d_1,d_2\in\mathbb{Z}_{>0}$ such that $d=d_1+d_2$, 
$\Sigma_1\subset\Sigma$ is a subfan whose every 
maximal cone is $d_1$-dimensional, while 
$\Sigma_2\subset\Sigma$ is a subfan whose every 
maximal cone is $d_2$-dimensional. 
Suppose further that 
the support 
\[
\left|\Sigma_1\right|:=\bigcup_{\sigma\in\Sigma_1}\sigma
\] 
of $\Sigma_1$ is a $d_1$-dimensional 
subspace in $N_{\mathbb{R}}$. 
Then, $X$ has an $X_1$-bundle structure 
induced by the surjective morphism $N_{\mathbb{R}}\to N_{\mathbb{R}}/|\Sigma_1|$, 
where $X_1$ is a smooth projective toric $d_1$-fold associated to 
the $d_1$-dimensional fan $\Sigma_1$ in $|\Sigma_1|$. 

In order to test the above condition, the following easy lemma is useful:
\begin{lem}\label{pcbdl}
$|\Sigma_1|\subset N_{\mathbb{R}}$ is a $d_1$-dimensional 
subspace if and only if for any primitive collection $P\subset G_1$, 
we have $\sigma(P)\cap\G(\Sigma)\subset G_1$. 
\end{lem}


\section{Toric special weak Fano manifolds} \label{specialfanosec}
Fano manifolds and weak Fano manifolds 
are fundamental objects in birational 
geometry, and are simply defined as follows:
\begin{defn}
Let $X=X_{\Sigma}$ be a smooth projective 
toric $d$-fold. We call $X$ a toric {\em Fano} 
(resp. {\em weak Fano}) 
manifold if its anticanonical 
divisor $-K_X$ is {\em ample} (resp. {\em nef}). 
By Definition \ref{pcpc} and 
Theorem \ref{toricconethm}, we have the 
following equivalence: 
$X$ is a Fano (resp. weak Fano) manifold if and only if 
$\deg P>0$ (resp. $\deg P\ge 0$) for any 
primitive collection $P$ in $\Sigma$.
\end{defn}

We introduce the following subclass of toric weak Fano manifolds 
in order to settle the characterization of toric weakened Fano manifolds.

\begin{defn}\label{specon}
Let $X=X_{\Sigma}$ be a smooth toric weak Fano $d$-fold, and 
$\varphi: X\to\overline{X}$ be a primitive crepant contraction. 
We call $\varphi$ a {\em special} primitive crepant contraction if the following hold:
\begin{enumerate}
\item $\varphi$ corresponds to an extremal primitive relation $x_1+x_2=2x$, 
where $x_1,x_2,x\in\G(\Sigma)$. Namely, 
$\varphi|_E:E\to \varphi(E)$ has a $\mathbb{P}^1$-bundle structure, where 
$E$ is the exceptional divisor of $\varphi$ and $\varphi(E)$ is 
a $(d-2)$-dimensional 
torus invariant submanifold in $\overline{X}$. 
\item $\varphi(E)$ is a smooth toric weak Fano $(d-2)$-fold. 
\item 
Let $\Sigma_{\varphi}\subset \Sigma$ be the $(d-2)$-dimensional subfan 
\[
\left\{\tau\in\Sigma\,\left|\,\left(\tau\cap\G(\Sigma)\right)\cap\{x,x_1,x_2\}=\emptyset,\mbox{ while }
\exists \sigma\in\Sigma\mbox{ s.t. }\tau\prec\sigma\mbox{ and }
x\in\sigma\cap\G(\Sigma)\right.\right\},
\]
where $\tau\prec\sigma$ means that $\tau$ is a face of $\sigma$. 
Then, there exists $m\in M$ such that $m(v)\ge -1$ for any $v\in\G(\Sigma)$, 
$m(x)=m(x_1)=m(x_2)=-1$ 
and $m(y)=0$ 
for any $y\in\G(\Sigma_{\varphi})$. 
\end{enumerate}
\end{defn}
We frequently use the following claim, though this is an easy consequence of 
the definition of $\Sigma_\varphi$:
\begin{claim}
$\langle x\rangle+\tau\in\Sigma$ for any $\tau\in\Sigma_\varphi$. 
\end{claim}

\begin{defn}\label{mizukawa}
Let $X=X_{\Sigma}$ be a smooth toric weak Fano $d$-fold. 
We call $X$ a toric {\em special} weak Fano $d$-fold if $X$ is {\em not} 
a Fano manifold, and 
every primitive crepant contraction $\varphi:X\to\overline{X}$ on $X$ is special. 
\end{defn}

\begin{rem}
The notion of {\em special} primitive crepant contractions is a generalization of 
$3$-dimensional $(0,2)$-type contractions 
studied by Minagawa (see \cite{minagawa} and \cite{sato2}). 
In particular,  a smooth toric weak Fano $3$-fold is special if and only if it is a 
toric weakened Fano $3$-fold. 
\end{rem}

In the rest of this section, we study the structure of 
toric special weak Fano manifolds.

The following proposition is well-known for experts of 
Minimal Model Program. 
We describe a sketch of the proof for the reader's convenience. 

\begin{prop}\label{thanksfujino}
Let $X$ be a quasi-projective canonical toric $d$-fold. 
If there exist two projective {\em crepant} toric morphisms 
$\varphi_1:X_1\to X$ and $\varphi_2:X_2\to X$ 
such that $X_1$ and $X_2$ are $\mathbb{Q}$-{\em factorial} and {\em terminal}, 
then $X_1$ and $X_2$ can be obtained from each 
other by a finite succession of {\em flops}. 
\end{prop}
\begin{proof}[Sketch of the proof]
Since $\varphi_1$ and $\varphi_2$ are crepant, 
$X_1$ and $X_2$ are isomorphic in codimension one. 
So, we can do the same argument in the proof of 
\cite[Proposition 5.7]{fujino-notes} by using relative toric Mori theory 
developed in \cite{fujino-sato}. 
\end{proof}

\begin{lem}\label{anticanonical}
 A smooth toric special weak Fano $d$-fold $X=X_{\Sigma}$ is 
 uniquely determined only by $\G(\Sigma)$. 
\end{lem}
\begin{proof}
The anticanonical model $\Phi_{\left|-K_X\right|}(X)$ 
is a Gorenstein toric Fano variety and it is uniquely determined 
by $\G(\Sigma)$. Every primitive crepant contraction of $X$ is 
special, that is, $X$ has no flopping contraction. Therefore, 
Proposition \ref{thanksfujino} tells us that 
the crepant resolution of  $\Phi_{|-K_X|}(X)$ is also unique. 
\end{proof}

\begin{rem}
Lemma \ref{anticanonical} does not hold for a general 
smooth toric weak Fano variety even in dimension $3$, 
because the Ordinary Double Point has two different crepant resolutions.
\end{rem}

\begin{prop}\label{keyprop}
Let $X=X_\Sigma$ be a smooth toric special weak Fano $d$-fold. 
Then, for any primitive collection $P=\{z_1,z_2\}\subset\G(\Sigma)$ such that 
$z_1+z_2\neq 0$, there exists $z\in\G(\Sigma)$ such that 
$z$ is contained in the relative interior of $\langle z_1,z_2\rangle$. 
\end{prop}
\begin{proof} 
If $\deg P=2$ then $\dim\sigma(P)=0$, while if $\deg P=1$ then $\dim\sigma(P)=1$. 
Thus, we may assume that 
$\deg P=0$ and $\dim\sigma(P)=2$, that is, 
$\sigma(P)=\langle x,y\rangle$ for $x,y\in\G(\Sigma)$ 
and the primitive relation for $P$ is $z_1+z_2=x+y$. 

Since every primitive crepant contraction is special, $P$ is not extremal. Therefore, 
there exist extremal primitive collections $P_1,\ldots, P_n\subset\G(\Sigma)$ such that 
\[
r(P)=a_1r(P_1)+\cdots+a_nr(P_n),
\]
where $n\ge 2$ and $a_1,\ldots,a_n>0$. Since $\deg P=0$, 
we have $\deg P_1=\cdots=\deg P_n=0$. This means that for any $1\le i\le n$, 
the extremal contraction associated to $P_i$ must be a special primitive crepant contraction. 
In particular, there exist two extremal primitive relations
\[
x_1+x_2=2x\mbox{ and }y_1+y_2=2y,
\]
where $\{x_1,x_2,y_1,y_2\}\subset\G(\Sigma)$. 
We remark that $\langle x,y\rangle
=\sigma(P)\in\Sigma$. 
Since some of elements $x,y,x_1,x_2,y_1,y_2$ may coincide, 
we have to consider the following four cases:
\begin{enumerate}
\item $x_1=y_1$.
\item $x=y_1$.
\item $x=y_1$ and $y=x_2$. 
\item $x,y,x_1,x_2,y_1,y_2$ are distinct elements. 
\end{enumerate}

\medskip

(1) The case where $x_1=y_1$. Since $x_1+x_2=2x$ and $y_1+y_2=2y$ are extremal, 
Proposition \ref{contractible} tells us that 
$\langle x,y_2\rangle, \langle x_2,y\rangle\in\Sigma$. 
On the other hand, 
the relations $x_1+x_2=2x$, $y_1+y_2=2y$ and $x_1=y_1$ imply that 
\[
x_1+x_2+2y=2x+y_1+y_2\ \Longleftrightarrow\ x_2+2y=2x+y_2.
\]
The last relation means that either $\{ x_2,y\}$ or $\{ x,y_2\}$ 
does not generate a cone in $\Sigma$. This is a contradiction. 

\medskip

(2) The case where $x=y_1$. 
By Proposition \ref{contractible}, $\langle y,x_1\rangle$ and $\langle y,x_2\rangle$ 
are in $\Sigma$ as above. 
Since $y_1+y_2=2y$ corresponds to a special primitive crepant contraction, 
there exists $m\in M$ as in Definition \ref{specon}. 
For this $m$, we have 
\[
m(y_1)=m(y_2)=m(y)=-1\mbox{ and }m(x_1)=m(x_2)=0.
\] 
On the other hand, 
\[
m(x_2)=m(2x-x_1)=2m(x)-m(x_1)=2m(y_1)-m(x_1)=-2<0, 
\]
a contradiction. 

\medskip

(3) The case where $x=y_1$ and $y=x_2$. 
Let $m\in M$ be the element as in Definition \ref{specon} for 
the extremal primitive relation $x_1+y=2x$. 
\[
m(z_1)+m(z_2)=m(x)+m(y)=-2
\]
implies  
\[
m(z_1)=m(z_2)=-1.
\] 
If $z_1$ and $z_2$ are not contained in 
the line passing through $x_1,x,y,y_2$, then 
there must exist $v\in\G(\Sigma)\setminus\{x_1,y,x\}$ such that 
$m(v)=-1$ and $\langle x,v\rangle\in\Sigma$, 
since the triangle whose vertices are $x_1$, $y$ and $z_1$ 
is contained in the hyperplane 
\[
\left\{u\in N_{\mathbb{R}}\,\left|\,m(u)=-1\right.\right\}\subset N_{\mathbb{R}}
\]
and $-1$ is the minimum value of $m$ on $\G(\Sigma)$.
However, this is impossible, because 
$\langle x,v\rangle\in\Sigma$ implies $m(v)=0$ by the definition of 
a special primitive crepant contraction. 
Therefore, $z_1,z_2,x_1,x,y,y_2$ are contained in a line. 
Since $z_1+z_2=x+y$ and $\langle x,y\rangle\in\Sigma$, 
in this case, $x$ and $y$ are contained in the relative 
interior of $\langle z_1,z_2\rangle$. 

\medskip

(4) The case where  $x_1,x_2,x,y_1,y_2,y$ are distinct elements. 
By applying Proposition \ref{contractible} to the extremal primitive relations 
$x_1+x_2=2x$ and $y_1+y_2=2y$ several times, we can easily see that 
\[
\langle x,y,x_1,y_1\rangle,\ \langle x,y,x_1,y_2\rangle,\ 
\langle x,y,x_2,y_1\rangle,\ \langle x,y,x_2,y_2\rangle\in\Sigma. 
\]
Let $m_1,m_2\in M$ be the elements associated to 
the extremal primitive relations $x_1+x_2=2x$ and $y_1+y_2=2y$ 
as in Definition \ref{specon}, respectively. 
We remark that 
\[
m_1(x_1)=m_1(x_2)=m_1(x)=m_2(y_1)=m_2(y_2)=m_2(y)=-1,
\]
while 
\[
m_1(y_1)=m_1(y_2)=m_1(y)=m_2(x_1)=m_2(x_2)=m_2(x)=0.
\]
Then, 
\[
m_1(z_1+z_2)=m_1(x+y)\ \Longleftrightarrow\ 
m_1(z_1)+m_1(z_2)=m_1(x)+m_1(y)=-1+0=-1
\]
says that either $m_1(z_1)$ or $m_1(z_2)$ is equal to $-1$. 
Without loss of generality, we may assume $m_1(z_1)=-1$ and $m_1(z_2)=0$. 
As in the case of (3), $z_1$ has to be contained in the line $L_1$ 
passing through $x_1,x_2,x$. 
Similarly, we obtain $m_2(z_1)+m_2(z_2)=m_2(x)+m_2(y)=0-1=-1$. 
Thus, $m_2(z_1)=-1$ or $m_2(z_2)=-1$. 
If $m_2(z_1)=-1$, then $z_1$ has to be contained in the line $L_2$ 
passing through $y_1,y_2,y$ as above. 
Since $L_1\cap L_2=\emptyset$, this is impossible. 
Thus, we have $m_2(z_2)=-1$ and 
$z_2,y_1,y_2,y$ are contained in a line. 

Consequently, $z_1$ and $z_2$ are expressed as follows: 
\[
z_1=sx_1+(1-s)x\ (s\neq 0),\mbox{ while }z_2=ty_1+(1-t)y\ (t\neq 0).
\]
Since $x,y,x_1,y_1$ are linearly independent, $x,y,z_1,z_2$ are also linearly independent. 
However, this is a contradiction, because there is a linear relation $z_1+z_2=x+y$. 
\end{proof}

Let $X=X_\Sigma$ be a smooth toric special weak Fano $d$-fold. 
Since $X$ is not Fano, there exists at least one special primitive crepant contraction 
$\varphi:X\to\overline{X}$. 
So, let $x_1+x_2=2x$ ($x_1,x_2,x\in\G(\Sigma)$) be the 
extremal primitive relation associated to $\varphi$. 
We use the same notation $E\subset X$, 
$\varphi(E)$ and $\Sigma_{\varphi}\subset\Sigma$ 
as in Definition \ref{specon}. 

Put 
\[
\mathcal{S}_1:=
\left(\mathbb{R}_{\ge 0}x+\mathbb{R}_{\ge 0}x_1\right)\cup
\left(\mathbb{R}_{\ge 0}x+\mathbb{R}_{\ge 0}x_2\right)\cup
\left(\mathbb{R}_{\ge 0}(-x)+\mathbb{R}_{\ge 0}x_1\right)\cup
\left(\mathbb{R}_{\ge 0}(-x)+\mathbb{R}_{\ge 0}x_2\right),
\]
while put 
\[
\mathcal{S}_2:=
\left(\bigcup_{\sigma\in\Sigma_{\varphi}}\left(\mathbb{R}_{\ge 0}x+\sigma\right)\right)\cup
\left(\bigcup_{\sigma\in\Sigma_{\varphi}}\left(\mathbb{R}_{\ge 0}(-x)+\sigma\right)\right).
\]
\begin{rem}
By Proposition \ref{contractible}, 
we have 
$\mathbb{R}_{\ge 0}x+\mathbb{R}_{\ge 0}x_1,
\mathbb{R}_{\ge 0}x+\mathbb{R}_{\ge 0}x_2\in\Sigma$, 
while $\mathbb{R}_{\ge 0}x+\sigma\in\Sigma$ for any $\sigma\in\Sigma_{\varphi}$, 
where $\mathbb{R}_{\ge 0}x+\sigma$ is a $(\dim\sigma+1)$-dimensional 
cone, because it is simplicial. 
We remark that $\mathcal{S}_1$ is a subspace in $N_\mathbb{R}$, 
but $\mathcal{S}_2$ is not necessarily a subspace. 
Moreover, $\mathcal{S}_1\cap\mathcal{S}_2=\mathbb{R}x$. 
\end{rem}
We can show the following: 
\begin{prop}\label{3nenagumi} 
Put 
\[
I:=\G(\Sigma)\setminus\left(\{x,x_1,x_2\}\cup\G(\Sigma_{\varphi})\right).
\]
Then, either $I\subset\mathcal{S}_1$ or $I\subset\mathcal{S}_2$ holds. 
\end{prop}
\begin{proof}
First, we show that $I\subset\mathcal{S}_1\cup\mathcal{S}_2$. 
Remark that for any $z\in I$, $\{x,z\}$ is a primitive collection. 
If $x+z=0$, then $z=-x\in\mathcal{S}_1\cap\mathcal{S}_2$. 
So suppose that $x+z\neq 0$. Then, 
Proposition \ref{keyprop} says that there exists 
$z'\in\G(\Sigma)$ such that $z'$ is contained 
in the relative interior of $\langle x,z\rangle$. 
If $z'=x_1$, then $ax+bz=x_1$ for $a,b>0$. 
So, we have $z=\frac{1}{b}(-ax+x_1)\in\mathcal{S}_1$. The case where 
$z'=x_2$ is completely similar. If $z'=y$ for some $y\in\G(\Sigma_{\varphi})$, then 
we can show that $z\in\mathcal{S}_2$ similarly. 
So, we may assume $z'\in I$. In this case, we can replace $z$ by $z'$ and 
do the same arguments repeatedly. This operation has to finish in finite steps. 
Therefore, we obtain a finite subset 
$\{z_0:=z,z_1,\ldots,z_s\}\subset \G(\Sigma)$ such that 
$z_0,\ldots,z_{s-1}\in I$, while $z_s\in \{x_1,x_2\}\cup\G(\Sigma_{\varphi})$. 
In any case, $\{z_0,z_1,\ldots,z_s\}$ is contained in either $\mathcal{S}_1$ or 
$\mathcal{S}_2$. 


Next, we show that $I\subset\mathcal{S}_1$ or $I\subset\mathcal{S}_2$. 
Suppose that there exist $z_1\in\mathcal{S}_1\cap I$ and 
$z_2\in\mathcal{S}_2\cap I$ such that $z_1\neq -x$ and $z_2\neq -x$. 
As mentioned above, without loss of generality, we may assume that 
$x_1$ is contained in the relative interior of $\langle x,z_1\rangle$, while 
for some $y\in\G(\Sigma_{\varphi})$, $y$ is contained in the relative interior of 
$\langle x,z_2\rangle$. Thus, we obtain 
\[
a_1x+b_1z_1=x_1\mbox{ and }a_2x+b_2z_2=y\mbox{ for some }a_1,a_2,b_1,b_2>0.
\]
Then, 
\[
\frac{b_1}{a_1}z_1+\frac{1}{a_2}y=\frac{1}{a_1}x_1+\frac{b_2}{a_2}z_2
\]
means that either $\{z_1,y\}$ or $\{x_1,z_2\}$ is a primitive collection 
because the intersection of the relative interiors of 
$\langle z_1,y\rangle$ and $\langle x_1,z_2\rangle$ 
is one-dimensional. 
Suppose that $\{z_1,y\}$ is a primitive collection, that is, 
the relative interior of $\{z_1,y\}$ contains $w\in\G(\Sigma)$ by Proposition \ref{keyprop}. 
We remark that 
$z_1\in\mathbb{R}x+\mathbb{R}_{\ge 0}x_1\subset\mathcal{S}_1$, 
while there exists a $(d-2)$-dimensional cone $\sigma\in\Sigma_\varphi$ 
such that $y\in \mathbb{R}x+\sigma\in\mathcal{S}_2$. 
This means 
$w\not\in\mathcal{S}_1\cup\mathcal{S}_2$, a contradiction. 
The other case is similar. 
\end{proof}

For the case $I\subset \mathcal{S}_1$, more precisely, we obtain the following.
\begin{prop}\label{ichiomain}
If $I\subset \mathcal{S}_1$, then
$X$ is isomorphic to a smooth 
toric weak del Pezzo surface bundle over $\varphi(E)$. 
\end{prop}
\begin{proof}
%
First, we remark that $\G(\Sigma)\cap\mathcal{S}_1=\{x,x_1,x_2\}\cup I$. 
For any $y\in\G(\Sigma_\varphi)$, we have $m(y)=0$ by the definition, 
so there must exist an element $z\in\G(\Sigma)\cap\mathcal{S}_1$ such that 
$m(z)>0$. More precisely, since $X$ is a smooth weak Fano variety, an easy observation 
tells us that $\G(\Sigma)\cap\mathcal{S}_1$ 
must contain $z_0:=-x$, $z_1:=-x_2$ or $z_2:=-x_1$. 

Suppose $z_0\in\G(\Sigma)\cap\mathcal{S}_1$. 
Proposition \ref{keyprop} says that $\{z_0,x_1\}$ generates a 
$2$-dimensional cone in $\Sigma$ or 
there exists an element in $\G(\Sigma)$ 
in the relative interior of $\langle z_0,x_1\rangle$. 
$\{z_0,x_2\}$ is similar.  By repeating this argument, 
we obtain the $2$-dimensional nonsingular complete fan $\Sigma'$ 
in the subspace $\mathcal{S}_1\subset N_\mathbb{R}$ such that 
$\G(\Sigma')=\{x,x_1,x_2\}\cup I$. The toric surface $S$ associated to $\Sigma'$ 
is a weak del Pezzo surface. Then, one can construct a fan $\Sigma^+$ 
such that $\G(\Sigma^+)=\G(\Sigma)$ and the associated toric variety is 
an $S$-bundle over $\varphi(E)$ (see Lemma \ref{pcbdl}). 
Then, Lemma \ref{anticanonical} tells us that $X^+\cong X$. 

So, suppose $z_0\not\in\G(\Sigma)\cap\mathcal{S}_1$. 
In this case, we may assume that $z_1\in\G(\Sigma)\cap\mathcal{S}_1$ 
without loss of generality. $z_3:=\frac{1}{2}(x_1+z_1)$ has to be 
contained in $\G(\Sigma)\cap\mathcal{S}_1$. 
Since $\G(\Sigma)\setminus\mathcal{S}_1=\G(\Sigma_\varphi)$, 
$x_1+z_1=2z_3$ has to be an extremal primitive relation such that 
the associated primitive crepant contraction $\varphi'$ is special. 
$\varphi'$ contracts its exceptional divisor to $\varphi(E)$. Therefore, 
there exists $m'\in M$ such that $m'(x_1)=m'(z_1)=m'(z_3)=-1$, 
while $m'(y)=0$ for any $y\in\G(\Sigma_\varphi)$. 
Thus, $\Sigma_\varphi$ is contained in the $(d-2)$-dimensional 
subspace in $N_\mathbb{R}$ defined by $m=0$ and $m'=0$. 
So, $\G(\Sigma)\cap\mathcal{S}_1$ 
must contain an element other than $x,x_1,x_2,z_1,z_3$. 
Similarly as the above case, we obtain 
the $2$-dimensional nonsingular complete fan $\Sigma'$ 
in $\mathcal{S}_1$. In this case, $X$ is isomorphic to 
the direct product of $S=S_{\Sigma'}$ and $\varphi(E)$. 
%
%
%
%
%
%
\end{proof}

\section{The classification of smooth toric special weak Fano 4-folds}
\label{sec4fold} 
In this section, we classify all the smooth 
toric special weak Fano $4$-folds 
using the results in the previous section. 
We omit the trivial case, that is, the case where 
a smooth toric special weak Fano $4$-fold $X$ is 
isomorphic to a direct product of lower-dimensional 
subvarieties. In this case, $X$ is trivially a smooth toric 
weakened Fano $4$-fold. 
We use the same notation as in Proposition \ref{3nenagumi} such as 
$x_1+x_2=2x$, $m$, $\mathcal{S}_1$, $\mathcal{S}_2$, $I$, 
$\varphi(E)$ and $\Sigma_{\varphi}$. 

For the classification, we describe the list of smooth 
toric weak del Pezzo surfaces in \cite{sato2} here. 
We should remark that this list is equivalent to the one 
of Gorenstein toric del Pezzo surfaces (see \cite{koelman}). 
Here, we use row vectors $w_i$ to describe elements in $N$. In this classification list, 
we denote by $F_a$ the Hirzebruch surface of degree $a$, 
while we denote by $S_n$ the del Pezzo surface of degree $n$.


\begin{table}[htb]
\begin{center}
\caption{Smooth toric weak del Pezzo surfaces.}
\begin{tabular}{|c|l|l|}
\hline
notation  & $\G(\Sigma)$   \\
\hline\hline

$\mathbb{P}^2$ & \begin{tabular}{l}
$w_1=(1,0),w_2=(0,1),w_3=(-1,-1)$ 
\end{tabular} \\ \hline
$\mathbb{P}^1\times\mathbb{P}^1$ & 
\begin{tabular}{l}
$w_1=(1,0),w_2=(0,1),w_3=(-1,0),w_4=(0,-1)$ 
\end{tabular}
\\ \hline
$F_1$ & \begin{tabular}{l}
$w_1=(1,0),w_2=(0,1),w_3=(-1,0),w_4=(1,-1)$ 
\end{tabular}
\\ \hline
$F_2$ & \begin{tabular}{l} $w_1=(1,0),w_2=(1,1),w_3=(-1,0),w_4=(1,-1)$ 
\end{tabular}
\\ \hline
$S_7$  & \begin{tabular}{l} $w_1=(1,0),w_2=(1,1),w_3=(0,1),w_4=(-1,0),
w_5=(0,-1)$ \end{tabular} \\ \hline
$W_3$ & \begin{tabular}{l} $w_1=(1,0),w_2=(1,1),w_3=(0,1),w_4=(-1,0),
w_5=(1,-1)$ \end{tabular}  \\ \hline
$S_6$ & \begin{tabular}{l} $w_1=(1,0),w_2=(1,1),w_3=(0,1),w_4=(-1,0),
w_5=(-1,-1),$ \\ $w_6=(0,-1)$ 
\end{tabular} \\ \hline
$W^1_4$ & \begin{tabular}{l} $w_1=(1,0),w_2=(1,1),w_3=(0,1),w_4=(-1,0),
w_5=(0,-1),$ \\ $w_6=(1,-1)$ \end{tabular} 
\\ \hline
$W^2_4$ & \begin{tabular}{l} $w_1=(1,0),w_2=(1,1),w_3=(0,1),w_4=(-1,1),
w_5=(-1,0),$ \\ $w_6=(1,-1)$ \end{tabular} 
\\ \hline
$W^3_4$ & \begin{tabular}{l} $w_1=(1,0),w_2=(1,1),w_3=(1,2),w_4=(0,1),
w_5=(-1,0),$ \\ $w_6=(1,-1)$ 
\end{tabular}  \\ \hline
$W^1_5$  & 
\begin{tabular}{l}
$w_1=(1,0),w_2=(1,1),w_3=(0,1),w_4=(-1,1),
w_5=(-1,0),$ \\ $w_6=(0,-1),w_7=(1,-1)$ 
\end{tabular} \\ \hline
$W^2_5$ & \begin{tabular}{l}
$w_1=(1,0),w_2=(1,1),w_3=(0,1),w_4=(-1,0),
w_5=(0,-1),$ \\ $w_6=(1,-2),w_7=(1,-1)$ 
\end{tabular}   \\ \hline
$W^1_6$ & \begin{tabular}{l}
$w_1=(1,0),w_2=(1,1),w_3=(0,1),w_4=(-1,1),
w_5=(-1,0),$ \\ $w_6=(-1,-1),
w_7=(0,-1),w_8=(1,-1)$ 
\end{tabular}   \\ \hline
$W^2_6$ & \begin{tabular}{l}
$w_1=(1,0),w_2=(1,1),w_3=(0,1),w_4=(-1,1),
w_5=(-1,0),$ \\ $w_6=(0,-1),
w_7=(1,-2),w_8=(1,-1)$ 
\end{tabular}   \\ \hline
$W^3_6$ & \begin{tabular}{l}
$w_1=(1,0),w_2=(1,1),w_3=(1,2),w_4=(0,1),
w_5=(-1,0),$ \\ $w_6=(0,-1),w_7=(1,-2),w_8=(1,-1)$ 
\end{tabular}  \\ \hline
$W_7$ & \begin{tabular}{l}
$w_1=(1,0),w_2=(1,1),w_3=(0,1),w_4=(-1,1),
w_5=(-2,1),$ \\ $w_6=(-1,0),w_7=(0,-1),w_8=(1,-2),
w_9=(1,-1)$ 
\end{tabular}  \\ \hline

\end{tabular}
\end{center}
\end{table}

We give the figures for $\mathbb{P}^2$, $\mathbb{P}^1\times\mathbb{P}^1$, 
$F_1$, 
$F_2$, $W_3$ and $W_4^1$ for the reader's convenience. 
These varieties are recurring throughout our classification.

\newpage

\begin{figure}[ht]
\centering
\begin{picture}(380,250)

\put(10,240){$\mathbb{P}^2$}

\put(30,160){\circle*{6}}
\put(10,147){$w_3$}

\put(90,190){\circle*{6}}
\put(97,188){$w_1$}

\put(55,230){$w_2$}
\put(60,220){\circle*{6}}

\put(30,160){\line(1,2){30}}
\put(30,160){\line(2,1){60}}
\put(60,220){\line(1,-1){30}}

\put(60,190){\circle{6}}

\put(140,240){$\mathbb{P}^1\times\mathbb{P}^1$}

\put(160,190){\circle*{6}}
\put(140,188){$w_3$}

\put(220,190){\circle*{6}}
\put(227,188){$w_1$}

\put(185,230){$w_2$}
\put(190,220){\circle*{6}}

\put(185,147){$w_4$}
\put(190,160){\circle*{6}}

\put(190,220){\line(1,-1){30}}
\put(190,220){\line(-1,-1){30}}

\put(190,160){\line(1,1){30}}
\put(190,160){\line(-1,1){30}}

\put(190,190){\circle{6}}

\put(270,240){$F_1$}

\put(290,190){\circle*{6}}
\put(270,188){$w_3$}

\put(350,190){\circle*{6}}
\put(357,188){$w_1$}

\put(315,230){$w_2$}
\put(320,220){\circle*{6}}

\put(357,147){$w_4$}
\put(350,160){\circle*{6}}

\put(320,220){\line(1,-1){30}}
\put(320,220){\line(-1,-1){30}}

\put(350,160){\line(0,1){30}}
\put(350,160){\line(-2,1){60}}

\put(320,190){\circle{6}}


\put(10,100){$F_2$}

\put(30,50){\line(2,1){60}}
\put(30,50){\line(2,-1){60}}
\put(90,80){\line(0,-1){60}}


\put(30,50){\circle*{6}}
\put(10,48){$w_3$}

\put(90,80){\circle*{6}}
\put(97,90){$w_2$}

\put(90,50){\circle*{6}}
\put(97,48){$w_1$}

\put(90,20){\circle*{6}}
\put(97,7){$w_4$}

\put(60,50){\circle{6}}


\put(140,100){$W_3$}

\put(160,50){\line(1,1){30}}
\put(190,80){\line(1,0){30}}
\put(160,50){\line(2,-1){60}}
\put(220,80){\line(0,-1){60}}


\put(160,50){\circle*{6}}
\put(220,80){\circle*{6}}
\put(220,50){\circle*{6}}
\put(220,20){\circle*{6}}
\put(190,80){\circle*{6}}

\put(185,90){$w_3$}

\put(140,48){$w_4$}

\put(227,90){$w_2$}

\put(227,48){$w_1$}

\put(227,7){$w_5$}

\put(190,50){\circle{6}}

\put(270,100){$W_4^1$}

\put(290,50){\line(1,1){30}}
\put(320,80){\line(1,0){30}}
\put(350,80){\line(0,-1){60}}
\put(350,20){\line(-1,0){30}}
\put(320,20){\line(-1,1){30}}


\put(290,50){\circle*{6}}
\put(350,80){\circle*{6}}
\put(350,50){\circle*{6}}
\put(350,20){\circle*{6}}
\put(320,80){\circle*{6}}
\put(320,20){\circle*{6}}

\put(320,50){\circle{6}}

\put(315,90){$w_3$}

\put(315,7){$w_5$}

\put(270,48){$w_4$}

\put(357,90){$w_2$}

\put(357,48){$w_1$}

\put(357,7){$w_6$}

\thicklines

\end{picture}
\caption{}
\label{baaiwake}
\end{figure}


\noindent
(I) the case where $I\subset\mathcal{S}_1$. 

Let $S$ be the smooth toric weak del Pezzo surface in $X$ 
corresponding to the fan in $\mathcal{S}_1$. 
In this case, $X$ is an $S$-bundle over $\varphi(E)$ (see Proposition \ref{ichiomain}). 
$\varphi(E)$ is also a smooth toric weak del Pezzo surface. 

\begin{rem}
In this case, for any primitive collection $P\subset\G(\Sigma)$, we have either 
$P\subset\G(\Sigma_\varphi)$ or $P\subset\G(\Sigma)\setminus\G(\Sigma_\varphi)$. 
Moreover, the primitive collections in $\G(\Sigma_\varphi)$ correspond to 
$\varphi(E)$, while the primitive collections in 
$\G(\Sigma)\setminus\G(\Sigma_\varphi)$ correspond to $S$. 
For example, let $X=X_\Sigma$ be the smooth toric special weak Fano $4$-fold of 
type $\mathcal{Z}_{10}$ below. Then, the primitive relations of $\Sigma$ are 
\[
y_1+y_3=w_3,\ y_2+y_4=y_1,
\]
\[
w_1+w_3=w_2,\ w_1+w_4=0,\ w_2+w_4=w_3,\ w_2+w_5=2w_1\mbox{ and }
w_3+w_5=w_1,
\]
where $\G(\Sigma)=\{y_1,y_2,y_3,y_4,w_1,w_2,w_3,w_4,w_5\}$ and 
$\G(\Sigma_\varphi)=\{y_1,y_2,y_3,y_4\}$. 
$\varphi(E)\cong F_1$, while $S\cong W_3$. 
The primitive relations of the $2$-dimensional fan $\Sigma_S$ associated to $S$ are 
\[
w_1+w_3=w_2,\ w_1+w_4=0,\ w_2+w_4=w_3,\ w_2+w_5=2w_1\mbox{ and }
w_3+w_5=w_1,
\]
where $\G(\Sigma_S)=\G(\Sigma)\setminus\G(\Sigma_\varphi)=
\{w_1,w_2,w_3,w_4,w_5\}$. On the other hand, 
by omitting $w_3$ from the above primitive relations, 
we obtain the primitive relations 
\[
y_1+y_3=0\mbox{ and }y_2+y_4=y_1
\]
of the $2$-dimensional fan $\Sigma'$ associated to $\varphi(E)$, 
where $\G(\Sigma')=\G(\Sigma_\varphi)=\{y_1,y_2,y_3,y_4\}$. 
\end{rem}

By Lemma \ref{pcbdl}, $X$ is not isomorphic to 
the direct product $S\times \varphi(E)$ 
if and only if 
there exists a primitive collection $P\subset\G(\Sigma_{\varphi})$ 
such that $\sigma(P)\cap\G(\Sigma)\not\subset\G(\Sigma_{\varphi})$. 

Suppose that $\rho(\varphi(E))\ge 2$, where $\rho(\varphi(E))=\#\G(\Sigma_\varphi)$ stands for 
the {\em Picard number} of $\varphi(E)$. In this case, every primitive collection 
of $\Sigma_{\varphi}$ contains exactly two elements. For $y\in N$, 
we denote by $\overline{y}$ the image of $y$ by 
the projection $N_{\mathbb{R}}\to N_{\mathbb{R}}/\mathcal{S}_1$, 
and we denote by $\overline{\Sigma}_\varphi$ 
the natural fan in $N_{\mathbb{R}}/\mathcal{S}_1$ associated to $\varphi(E)$. 
Let $P=\{y_1,y_2\}\subset\G(\Sigma_{\varphi})$ be a primitive collection, 
and assume that $\sigma(P)\cap\G(\Sigma)\not\subset\G(\Sigma_{\varphi})$. 
We remark that $\{\overline{y}_1,\overline{y}_2\}$ is also 
a primitive collection of $\overline{\Sigma}_\varphi$. 
If the primitive relation of $\{\overline{y}_1,\overline{y}_2\}$ is 
either 
\[
\overline{y}_1+\overline{y}_2=2\overline{y}_3\mbox{ or }
\overline{y}_1+\overline{y}_2=\overline{y}_3+\overline{y}_4, 
\]
where $\overline{y}_3,\overline{y}_4\in\G(\overline{\Sigma}_\varphi)$, 
then the primitive relation of $P$ is also either 
\[
y_1+y_2=2y_3\mbox{ or }
y_1+y_2=y_3+y_4,
\]
because $X$ is a weak Fano manifold, that is, 
$\deg P\ge 0$. This contradicts the assumption 
$\sigma(P)\cap\G(\Sigma)\not\subset\G(\Sigma_{\varphi})$. Next, 
suppose that the primitive relation of $\{\overline{y}_1,\overline{y}_2\}$ is 
\[
\overline{y}_1+\overline{y}_2=\overline{y}_3. 
\]
We may assume this primitive relation is extremal, that is, 
$\langle\overline{y}_1,\overline{y}_3\rangle$ and 
$\langle\overline{y}_3,\overline{y}_2\rangle$ are $2$-dimensional 
cones in $\overline{\Sigma}_\varphi$. 
Then, for some element $z\not\in\G(\Sigma_{\varphi})$, the primitive relation of $P$ 
has to be 
\[
y_1+y_2=y_3+z
\]
and 
$\{y_1,y_2,y_3,z\}$ spans a $3$-dimensional subspace in $N_{\mathbb{R}}$. 
On the other hand, there exists an element in $\G(\Sigma)$ in the relative interior 
of $\langle y_1,y_2\rangle$ by Proposition \ref{keyprop}. 
This is impossible. So, the remaining case is that 
the primitive relation of $\{\overline{y}_1,\overline{y}_2\}$ is 
\[
\overline{y}_1+\overline{y}_2=0.
\]
Let us show that $\{\overline{y}_1,\overline{y}_2\}$ is extremal. 
If $\{\overline{y}_1,\overline{y}_2\}$ is not extremal, then there 
exist extremal primitive collection $P_1,\ldots,P_l$ of $\overline{\Sigma}_\varphi$ $(l\ge 2)$ 
such that 
$r(\{\overline{y}_1,\overline{y}_2\})=a_1r(P_1)+\cdots+a_lr(P_l)$ 
for $a_1,\ldots,a_l\in\mathbb{Z}_{>0}$ 
(see Remark \ref{odasense}). Obviously, $\deg P_i\le 1$ 
for any $i$. This means that the primitive relation of $P_i$ 
in $\Sigma$ is same as the one in $\overline{\Sigma}_{\varphi}$ 
for any $i$ by the above argument. 
This is impossible because the primitive relation of 
$\{\overline{y}_1,\overline{y}_2\}$ in $\Sigma$ and 
the one in $\overline{\Sigma}_\varphi$ are distinct. 
Thus, we obtain the following:

\begin{prop}
Let $X$ be a smooth toric special weak Fano $4$-fold which 
is isomorphic to an $S$-bundle over $\varphi(E)$, where 
$S$ and $\varphi(E)$ are smooth toric weak del Pezzo surfaces, 
and $\rho(\varphi(E))\ge 2$. 
If $X\not\cong S\times \varphi(E)$, then $\varphi(E)$ has a Fano contraction, 
that is, a $\mathbb{P}^1$-bundle structure. 
\end{prop}


On the other hand, for the possiblities for the fiber $S$, 
the following holds. This fact makes our 
classification easier.

\begin{prop}\label{nihonbashi}
If $X$ is not isomorphic to the direct product 
$S\times \varphi(E)$, then $S$ is one of the following: $F_2$, $W_3$ and $W_4^1$. 
\end{prop}

Proposition \ref{nihonbashi} is a 
consequence of 
the following lemma and Table $1$. 
Remark that $S$ admits at least one crepant contraction. 

\begin{lem}
Suppose that there exist two special primitive 
crepant contractions associated to 
the extremal primitive relations 
$x_1+x_2=2x$ and 
$x'_1+x'_2=2x'$, where 
$x_1,x_2,x,x'_1,x'_2,x'\in\G(\Sigma)\cap
\mathcal{S}_1$ such that associated elements 
$m,m'\in M$ as in Definition \ref{specon} are linearly independent. 
Then, $X\cong S\times \varphi(E)$.
\end{lem}
\begin{proof}
Put $\G(\Sigma_{\varphi})=\{y_1,\ldots,y_l\}$. 
Then, we have 
\[
m(y_1)=\cdots=m(y_l)=m'(y_1)=\cdots=m'(y_l)=0.
\]
This means that $y_1,\ldots,y_l$ are 
contained in a $2$-dimensional subspace 
in $N_{\mathbb{R}}$. Namely, $X$ is isomorphic to 
the direct product of $S$ and $\varphi(E)$. 
\end{proof}

Thus, 
the possibilities for $\varphi(E)$ are $\mathbb{P}^2$, 
$\mathbb{P}^1\times  \mathbb{P}^1$, 
$F_1$ and $F_2$ by Table $1$, while 
the possibilities for $S$ are $F_2$, 
$W_3$ and $W_4^1$ by Proposition \ref{nihonbashi}. 
Here, for 
the reader's convenience, 
we explicitly describe the values $m(w_i)$ 
for the element $m\in M$ associated to a 
special primitive crepant contraction 
(we use the same notation as in Table $1$):
\begin{itemize}
\item $S\cong F_2$: The primitive relation associated 
to the special primitive crepant contraction is $w_2+w_4=2w_1$. 
$m(w_1)=m(w_2)=m(w_4)=-1$ and $m(w_3)=1$. 
\item $S\cong W_3$: The primitive relation associated 
to the special primitive crepant contraction is $w_2+w_5=2w_1$. 
$m(w_1)=m(w_2)=m(w_5)=-1$, $m(w_3)=0$ and $m(w_4)=1$. 
\item $S\cong W_4^1$:  The primitive relation associated 
to the special primitive crepant contraction is $w_2+w_6=2w_1$. 
$m(w_1)=m(w_2)=m(w_6)=-1$, $m(w_3)=m(w_5)=0$ and $m(w_4)=1$. 
\end{itemize}

\medskip

\underline{$\varphi(E)\cong\mathbb{P}^2$}\quad 
Let $\G(\Sigma_{\varphi})=\{y_1,y_2,y_3\}$. 
In this case, $\{y_1,y_2,y_3\}$ is the unique primitive collection 
in $\Sigma_{\varphi}$, and we have the primitive relation 
\[
\overline{y}_1+\overline{y}_2+\overline{y}_3=0
\]
in $\overline{\Sigma}_\varphi$. Obviously, $\{y_1,y_2,y_3\}$ is an extremal 
primitive collection. So, the possibilities for its primitive relation are 
\[
y_1+y_2+y_3=az_1\mbox{ and }y_1+y_2+y_3=z_1+z_2
\]
for $z_1,z_2\in\G(\Sigma)\setminus\G(\Sigma_{\varphi})$ 
and $a=1,2$ because 
the associated extremal ray has to be $K_X$-negative 
(otherwise, this primitive crepant contraction would be non-special). 
For the first case, $am(z_1)=m(y_1+y_2+y_3)=m(y_1)+m(y_2)+m(y_3)=0$. 
For the second case, similarly, we have $m(z_1)+m(z_2)=0$. 
However, since $m(z_1)$ and $m(z_2)$ are at least $-1$, we have 
three cases $(m(z_1),m(z_2))=(0,0)$, 
$(m(z_1),m(z_2))=(-1,1)$ and $(m(z_1),m(z_2))=(1,-1)$. 
The first case does not occur because 
$\langle z_1,z_2\rangle\in\Sigma$. 

By using the above arguments and Table $1$, we can easily classify all the 
smooth toric special weak Fano $4$-folds as follows. In the list, 
we use the notation in Table $1$ for $S$. $X$ is determined by giving 
the primitive relation for $\{y_1,y_2,y_3\}$. 


\begin{table}[htb]
\begin{center}
\caption{The case where $I\subset\mathcal{S}_1$ and $\varphi(E)\cong{\mathbb{P}}^2$.}
\begin{tabular}{|r||c||c|c|c||c|c|}
\hline
notation  & $\mathcal{Z}_1$ & $\mathcal{Z}_2$ & $\mathcal{Z}_3$ & $\mathcal{Z}_4$ & $\mathcal{Z}_5$ & $\mathcal{Z}_6$ \\
\hline\hline

$S\cong$ & $F_2$ & $W_3$ & $W_3$ & $W_3$ & $W_4^1$ & $W_4^1$ \\ \hline
$y_1+y_2+y_3=$ & $w_2+w_3$ & $w_3$ & $2w_3$ & $w_4+w_5$ 
& $w_3$ & $2w_3$ \\ \hline

\end{tabular}
\end{center}
\end{table}

\underline{$\varphi(E)\cong \mathbb{P}^1\times  \mathbb{P}^1$, $F_1$ or $F_2$}\quad 
For any extremal primitive relation $\overline{y}_1+\overline{y}_2=0$ in 
$\overline{\Sigma}_\varphi$, $\{y_1,y_2\}$ is also extremal, because 
$\{y_1,y_2\}\cap P=\emptyset$ for any primitive collection $P\neq\{y_1,y_2\}$ 
of $\Sigma$. So, if the associated contraction morphism is crepant, 
then it has to be special. Thus, 
the possibilities for the primitive relation of $\{y_1,y_2\}$ 
are 
\[
y_1+y_2=z\mbox{ and }y_1+y_2=2z,
\] 
where $z\in\G(\Sigma)\setminus\G(\Sigma_{\varphi})$ and $m(z)=0$. In the second case, 
there exists $m'\in M$ associated to this primitive relation. 
For $z_1,z_2\not\in\G(\Sigma_{\varphi})$ such 
that $\langle z_1,z\rangle,\langle z,z_2\rangle\in\Sigma$, 
we have $m'(z_1)=m'(z_2)=0$. Since $z$, $z_1$ and $z_2$ are contained in a 
$2$-dimensional subspace, $z_1$ and $z_2$ have to be 
centrally symmetric, that is, $z_1+z_2=0$. 
However, this phenomenon does not occur (see the classification below).

By using these results, we can do the classification easily. 
The classification lists are as follows.

\medskip

Let $\varphi(E)\cong\mathbb{P}^1\times\mathbb{P}^1$, and 
$\overline{y}_1+\overline{y}_3=0,\overline{y}_2+\overline{y}_4=0$ 
the primitive relations for $\varphi(E)$. Then, 
$X$ is determined by giving the primitive relations for 
$\{y_1,y_3\}$ and $\{y_2,y_4\}$. We remark that if either 
$y_1+y_3=0$ or $y_2+y_4=0$ holds, then $X$ is isomorphic to 
$\mathbb{P}^1\times V$ for a smooth toric special weak Fano $3$-fold 
$V$. As before, we omit this trivial case. 

\newpage

\begin{table}[htb]
\begin{center}
\caption{The case where $I\subset\mathcal{S}_1$ and $\varphi(E)\cong{\mathbb{P}}^1
\times\mathbb{P}^1$.}
\begin{tabular}{|r||c||c|c|}
\hline
notation  &   $\mathcal{Z}_7$ & $\mathcal{Z}_8$ & $\mathcal{Z}_{9}$   \\
\hline\hline

$S\cong$  & $W_3$  & $W_4^1$ & $W_4^1$  \\ \hline
$y_1+y_3=$  &  $w_3$ & $w_3$ 
& $w_3$  \\ \hline
$y_2+y_4=$  & $w_3$  & $w_3$ 
& $w_5$  \\ \hline

\end{tabular}
\end{center}
\end{table}


Let $\varphi(E)\cong F_1$, and 
$\overline{y}_1+\overline{y}_3=0,\overline{y}_2+\overline{y}_4=\overline{y}_1$ 
the primitive relations for $\varphi(E)$. Then, 
$X$ is determined by giving the primitive relation for 
$\{y_1,y_3\}$.

\begin{table}[htb]
\begin{center}
\caption{The case where $I\subset\mathcal{S}_1$ and $\varphi(E)\cong F_1$.}
\begin{tabular}{|r||c||c|}
\hline
notation  & $\mathcal{Z}_{10}$ & $\mathcal{Z}_{11}$   \\
\hline\hline

$S\cong$ & $W_3$ & $W_4^1$   \\ \hline
$y_1+y_3=$ & $w_3$ & $w_3$   \\ \hline

\end{tabular}
\end{center}
\end{table}

Let $\varphi(E)\cong F_2$, and 
$\overline{y}_1+\overline{y}_3=0,\overline{y}_2+\overline{y}_4=2\overline{y}_1$ 
the primitive relations for $\varphi(E)$. Then, 
$X$ is determined by giving the primitive relation for 
$\{y_1,y_3\}$.


\begin{table}[htb]
\begin{center}
\caption{The case where $I\subset\mathcal{S}_1$ and $\varphi(E)\cong F_2$.}
\begin{tabular}{|r||c||c|}
\hline
notation  & $\mathcal{Z}_{12}$ & $\mathcal{Z}_{13}$   \\
\hline\hline

$S\cong$ & $W_3$ & $W_4^1$   \\ \hline
$y_1+y_3=$ & $w_3$ & $w_3$   \\ \hline

\end{tabular}
\end{center}
\end{table}

\begin{rem}
$\mathcal{Z}_{12}$ and $\mathcal{Z}_{13}$ have 
two special primitive crepant contractions.
\end{rem}


\noindent
(II) the case where $I\subset\mathcal{S}_2$. 

First, we remark that for any $v\in\G(\Sigma)$, 
\[
 m(v) \left\{ \begin{array}{lcl}
    =-1 &  \cdots & v\in\{x_1,x_2,x\} \\
    =0 & \cdots & v\in\G(\Sigma_{\varphi}) \\
    \ge 1 & \cdots & \mbox{otherwise }(v\in I)
  \end{array} \right.
\]
because if $v\in I$, then $v\in\mathbb{R}_{\ge 0}(-x)+\sigma
\subset\mathcal{S}_2$ 
for some $\sigma\in\Sigma_{\varphi}$ 
by the definition of $\mathcal{S}_2$.  

Suppose that $\rho(\varphi(E))\ge 2$. Then, for any primitive collection 
$\{y_1,y_2\}\subset\G(\Sigma_{\varphi})$, 
an element $z\in\sigma(\{y_1,y_2\})\cap\G(\Sigma)$ has to have 
$m(z)=0$ by Proposition \ref{keyprop}. Namely, $z\in\G(\Sigma_{\varphi})$. 
So, $|\Sigma_{\varphi}|$ is a subspace in $N_{\mathbb{R}}$, and then 
$\mathcal{S}_2$ is also a subspace in $N_{\mathbb{R}}$. 
Thus, we obtain the following.
\begin{prop}\label{s2str}
Let $X$ be a smooth toric special weak Fano $4$-fold, 
and suppose that $I\subset\mathcal{S}_2$. 
If $\rho(\varphi(E))\ge 2$, then $X$ is isomorphic to a $V$-bundle 
over $\mathbb{P}^1$, where $V$ is the $3$-dimensional 
subvariety corresponding to the complete fan in $\mathcal{S}_2$. 
\end{prop}
\begin{rem}
In Proposition \ref{s2str}, $V$ has to be either 
a smooth toric Fano $3$-fold or 
a smooth toric special weak Fano $3$-fold. 
In the case of dimension $3$, it was shown by \cite{minagawa} 
and \cite{minagawaweakened} that 
a smooth toric weak Fano $3$-fold is special if and only if it is a 
weakened Fano manifold (also see \cite{sato2}). 
Moreover, if $V$ is a direct product of $\mathbb{P}^1$ and 
a smooth toric weak del Pezzo surface, then $X$ 
is isomorphic to either a direct product of $\mathbb{P}^1$ 
and a smooth toric special weak Fano $3$-fold or a direct 
product of $F_2$ and a smooth toric weak del Pezzo surface. 
We therefore omit these trivial cases.
\end{rem}

Now, we start the classification. First we consider the case where 
$X$ is isomorphic to a $V$-bundle over $\mathbb{P}^1$. 
$V$ have a toric fibration 
\[
f:V\to\mathbb{P}^1
\] 
associated to the map $\mathcal{S}_2\to\mathcal{S}_2/|\Sigma_\varphi|$ 
such that $V_t:=f^{-1}(t)\cong \varphi(E)$ for any $t\in\mathbb{C}$. 
The fiber $V_0$ is 
the torus invariant prime divisor associated to 
the cone $\mathbb{R}_{\ge 0}x$ in the subfan $\Sigma_V$ corresponding to $V$, 
while any irreducible component of $V_\infty$ 
is associated to an element in 
\[
\G(\Sigma_V)\setminus\left(\{x\}\cup\G(\Sigma_\varphi)\right).
\] 
If $V_{\infty}$ is also isomorphic to 
$\varphi(E)$, that is, 
\[
\#\left(\G(\Sigma_V)\setminus\left(\{x\}\cup\G(\Sigma_\varphi)\right)\right)=1,
\]
then $V$ is a $\varphi(E)$-bundle over $\mathbb{P}^1$. 


\underline{The case where $V$ is a smooth toric Fano $3$-fold.} 
We use the notation in the classification list of 
smooth toric Fano $3$-folds in \cite[Propositions 2.5.5, 2.5.6, 2.5.8 and 2.5.9]{bat2}. 
Also, we use the notation $\G(\Sigma_V)=\{v_1,v_2,\ldots\}$, 
thus we have $\G(\Sigma)=\{x_1,x_2,v_1,v_2,\ldots,\}$. 
We remark that $|\Sigma_V|=\mathcal{S}_2$. 
Trivially, $V\cong \mathbb{P}^3$ is impossible. 
\begin{itemize}
\item $V$ is of type $\mathcal{B}$. $\mathcal{B}_4$ is a direct product. 
For $\mathcal{B}_1$ and $\mathcal{B}_2$, there is no 
subvariety which corresponds to $\varphi(E)$. So, the remaining case is 
$\mathcal{B}_3$. The primitive relations of $\mathcal{B}_3$ are 
\[
v_1+v_2=v_3\mbox{ and }v_3+v_4+v_5=0,
\]
where $\G(\Sigma_V)=\{v_1,v_2,v_3,v_4,v_5\}$. 
$f:V\to\mathbb{P}^1$ is induced by 
\[
\mathcal{S}_2\to\mathcal{S}_2
/\left(\mathbb{R}v_3+\mathbb{R}v_4+\mathbb{R}v_5\right).
\] 
So, the possibilities for $x$ are $v_1$ and $v_2$. 
By adding one more primitive relation
\[
x_1+x_2=2v_1,
\]
where $\G(\Sigma)=\{x_1,x_2\}\cup\G(\Sigma_V)$, 
we obtain a new smooth toric special weak Fano $4$-fold $\mathcal{Z}_{14}$. 
The last primitive relation corresponds to a special 
primitive crepant contraction. $\varphi(E)\cong\mathbb{P}^2$ and 
$V$ is a $\varphi(E)$-bundle over $\mathbb{P}^1$. 
In the case where $x=v_2$, we obtain $\mathcal{Z}_{14}$, too. 

\item 
$V$ is of type $\mathcal{C}$. 
$\mathcal{C}_3$ and $\mathcal{C}_4$ is isomorphic to 
$\mathbb{P}^1\times\mathbb{P}^1\times\mathbb{P}^1$ and 
$\mathbb{P}^1\times F_1$, respectively. The primitive relations 
of the remaining cases are as follows: $v_1+v_2=0,$
\begin{table}[htb]
\begin{center}
\begin{tabular}{|r||c|c|c|}
\hline
notation  & $\mathcal{C}_1$ & $\mathcal{C}_2$ & $\mathcal{C}_5$   \\
\hline\hline
$v_3+v_5=$ & $v_1$ & $v_1$ & $v_1$  \\ \hline
$v_4+v_6=$ & $v_1$ & $v_3$  & $v_2$ \\ \hline
\end{tabular}
\end{center}
\end{table}

\noindent
where $\G(\Sigma_V)=\{v_1,v_2,v_3,v_4,v_5,v_6\}$. 
One can easily see that there exist exactly two, one and two toric surface bundle 
structures 
over $\mathbb{P}^1$ for 
$\mathcal{C}_1$, $\mathcal{C}_2$ and $\mathcal{C}_5$, respectively. 
They are induced by the following morphisms: 
\[
\mathcal{S}_2\to\mathcal{S}_2
/\left(\mathbb{R}v_1+\mathbb{R}v_2+\mathbb{R}v_3+\mathbb{R}v_5\right)\ 
(\mbox{for }\mathcal{C}_1,\ \mathcal{C}_2\mbox{ and }\mathcal{C}_5),
\]
\[
\mathcal{S}_2\to\mathcal{S}_2
/\left(\mathbb{R}v_1+\mathbb{R}v_2+\mathbb{R}v_4+\mathbb{R}v_6\right)\ 
(\mbox{for }\mathcal{C}_1\mbox{ and }\mathcal{C}_5).
\]
Thus, we have 
three new smooth toric special weak Fano $4$-folds $\mathcal{Z}_{15}$ 
$(V\cong\mathcal{C}_1)$, $\mathcal{Z}_{16}$ $(V\cong\mathcal{C}_2)$ and 
$\mathcal{Z}_{17}$ $(V\cong\mathcal{C}_5)$, by adding a primitve relation 
\[
x_1+x_2=2v_4
\]
to the above three cases, where $\G(\Sigma)=\{x_1,x_2\}\cup\G(\Sigma_V)$. 
This primitive relation corresponds to a special primitive crepant contraction. 
In every case, we have $\varphi(E)\cong F_1$.

\item $V$ is of type $\mathcal{D}$. The type $\mathcal{D}_1$ is impossible, 
because there does not exist a subvariety corresponding to $\varphi(E)$. 
For $V\cong\mathcal{D}_2$ whose primitive relations are 
\[
v_3+v_6=0,\ v_4+v_6=v_5,\ v_3+v_5=v_4,\ v_1+v_2+v_4=v_3\mbox{ and }
v_1+v_2+v_5=0,
\]
there exists exactly one 
toric fibration to $\mathbb{P}^1$ induced by the morphism
\[
\mathcal{S}_2\to\mathcal{S}_2
/\left(\mathbb{R}v_1+\mathbb{R}v_2+\mathbb{R}v_5\right).
\]
Since there exist three elements $v_3$, $v_4$ and $v_6$ other than $\{v_1,v_2,v_5\}$, 
$V_\infty$ is not isomorphic to $\mathbb{P}^2$, that is, $f$ is not a bundle structure. 
Thus, the possibility for $x$ is only $v_6$, and 
we obtain a smooth toric special weak Fano $4$-fold $\mathcal{Z}_{18}$ with one more 
primitve relation
\[
x_1+x_2=2v_6,
\]
where $\G(\Sigma)=\{v_1,v_2,v_3,v_4,v_5,v_6,x_1,x_2\}$. 
In this case, $\varphi(E)\cong\mathbb{P}^2$. 

\item $V$ is of type $\mathcal{E}$. 
The type $\mathcal{E}$ is a $S_7$-bundle over $\mathbb{P}^1$. In particular, 
$\mathcal{E}_3$ is isomorphic to $\mathbb{P}^1\times S_7$. 
The primitive relations 
of the remaining cases are as follows: $v_2+v_4=0,v_3+v_5=0,v_1+v_3=v_2,
v_2+v_5=v_1,v_1+v_4=v_5,$
\begin{table}[htb]
\begin{center}
\begin{tabular}{|r||c|c|c|}
\hline
notation  & $\mathcal{E}_1$ & $\mathcal{E}_2$ & $\mathcal{E}_4$   \\
\hline\hline
$v_6+v_7=$ & $v_1$ & $v_2$ & $v_3$  \\ \hline

\end{tabular}
\end{center}
\end{table}

\noindent
where $\G(\Sigma_V)=\{v_1,v_2,v_3,v_4,v_5,v_6,v_7\}$. 

For $V\cong\mathcal{E}_1$, 
there exists one toric fibration $V\to\mathbb{P}^1$ associated to 
the projection 
\[
\mathcal{S}_2\to\mathcal{S}_2/\left(\mathbb{R}v_1+\mathbb{R}v_2+
\mathbb{R}v_3+\mathbb{R}v_4+\mathbb{R}v_5\right).
\]
In fact, 
we obtain a smooth special weak Fano $4$-fold by adding the primitive relation 
$x_1+x_2=2v_6$, where $\G(\Sigma)=\G(\Sigma_V)\cup\{x_1,x_2\}$. 
In this case, $\varphi(E)\cong S_7$. 
We denote by $\mathcal{Z}_{19}$ this smooth toric special weak Fano $4$-fold. 

For $V\cong\mathcal{E}_2$, 
there exist two toric fibrations $V\to\mathbb{P}^1$ associated to 
the projections 
\[
\mathcal{S}_2\to\mathcal{S}_2/\left(\mathbb{R}v_1+\mathbb{R}v_2+
\mathbb{R}v_3+\mathbb{R}v_4+\mathbb{R}v_5\right)\mbox{ and }
\mathcal{S}_2\to\mathcal{S}_2/\left(\mathbb{R}v_2+\mathbb{R}v_4+
\mathbb{R}v_6+\mathbb{R}v_7\right).
\]
The first one is 
the case where $\varphi(E)\cong S_7$ and $x_1+x_2=2v_6$, while the second one 
is the case where $\varphi(E)\cong F_1$ and $x_1+x_2=2v_3$, 
where $\G(\Sigma)=\G(\Sigma_V)\cup\{x_1,x_2\}$. 
We denote by $\mathcal{Z}_{20}$ and $\mathcal{Z}_{21}$ these smooth toric special weak Fano $4$-folds, 
respectively. For $\mathcal{Z}_{21}$, the toric fibration $V\to\mathbb{P}^1$ 
has a singular fiber. Namely, $V$ is not an $F_1$-bundle over $\mathbb{P}^1$.  

For $V\cong\mathcal{E}_4$, we have two toric fibrations $V\to\mathbb{P}^1$, too. 
They are associated to the projections
\[
\mathcal{S}_2\to\mathcal{S}_2/\left(\mathbb{R}v_1+\mathbb{R}v_2+
\mathbb{R}v_3+\mathbb{R}v_4+\mathbb{R}v_5\right)\mbox{ and }
\mathcal{S}_2\to\mathcal{S}_2/\left(\mathbb{R}v_3+\mathbb{R}v_5+
\mathbb{R}v_6+\mathbb{R}v_7\right).
\]
The first one is 
the case where $\varphi(E)\cong S_7$ and $x_1+x_2=2v_6$, while the second one 
is the case where $\varphi(E)\cong F_1$ and $x_1+x_2=2v_4$, 
where $\G(\Sigma)=\G(\Sigma_V)\cup\{x_1,x_2\}$. 
We denote by $\mathcal{Z}_{22}$ and $\mathcal{Z}_{23}$ 
these smooth toric special weak Fano $4$-folds, 
respectively. 
We remark that for $\mathcal{Z}_{23}$, 
$V$ is not an $F_1$-bundle over $\mathbb{P}^1$.

\item $V$ is of type $\mathcal{F}$. 
The type $\mathcal{F}$ is a $S_6$-bundle over $\mathbb{P}^1$. In particular, 
$\mathcal{F}_1$ is isomorphic to $\mathbb{P}^1\times S_6$. 
So, let $V\cong\mathcal{F}_2$. 
The primitive relations of $V$ are
\[
v_1+v_3=v_2,\ v_1+v_4=0,\ v_1+v_5=v_6,\ v_2+v_4=v_3,\ v_2+v_5=0,
\]
\[
v_2+v_6=v_1,\ v_3+v_5=v_4,\ v_3+v_6=0,\ v_4+v_6=v_5\mbox{ and }
v_7+v_8=v_1,
\]
where $\G(\Sigma_V)=\{v_1,v_2,v_3,v_4,v_5,v_6,v_7,v_8\}$. 
Thus, we obtain a new smooth toric special weak Fano $4$-fold 
$\mathcal{Z}_{24}$ by adding the primitive relation $x_1+x_2=2v_7$, 
where $\G(\Sigma)=\G(\Sigma_V)\cup\{x_1,x_2\}$. 
In this case, $\varphi(E)\cong S_6$. 
We remark that there exist other 
three toric fibrations $V\to\mathbb{P}^1$ 
whose general fibers are isomorphic to $\mathbb{P}^2$. 
However, they are impossible, 
since 
there is no smooth fiber that could correspond to $V_0$.

\end{itemize}

\underline{The case where $V$ is a smooth toric weakned Fano $3$-fold.} 

We describe the four toric weakened Fano $3$-folds 
$X_3^0,X_4^0,X_4^1,X_5^1$, which are not isomorphic to direct products 
of lower-dimensional varieties as in \cite{sato2}. 
 
\medskip

Put
\[
x_0=(1,0,0),\ x_+=(1,1,0),\ x_-=(1,-1,0),\ y_+=(0,0,1),\ y_-=(0,0,-1),
\]
\[
z_1=(-1,0,1),\ z_2=(-1,0,0).
\]
Then, we have $V\cong X_3^0$ if $\G(\Sigma_V)=\{x_0,x_+,x_-,y_+,y_-,
z_1\}$, while $V\cong X_4^0$ if $\G(\Sigma_V)=\{x_0,x_+,x_-,y_+,y_-,
z_1,z_2\}$ (see Lemma \ref{anticanonical}). 

For $V\cong X_3^0$, 
there exists one toric fibration from $V$ to $\mathbb{P}^1$ associated to the projection 
\[
\mathcal{S}_2\to\mathcal{S}_2/\left(\mathbb{R}x_0+\mathbb{R}y_+
+\mathbb{R}y_-+\mathbb{R}z_1\right).
\]
However, the smooth toric 
weak Fano $4$-fold $X$ obtained by adding a primitive relation 
$x_1+x_2=2x_+$ is not special, because the primitive crepant 
contraction associated to $x_++x_-=2x_0$ is not special. 
In fact, 
for $m'\in M$ associated to $x_++x_-=2x_0$ as in Definition \ref{specon}, 
we have $m'(x_+)=m'(x_-)=m'(x_0)=-1$, while 
$m'(y_+)=m'(y_-)=0$, $m'(x_1)\ge 0$ and $m'(x_2)\ge 0$. 
This contradicts the primitive relation $x_1+x_2=2x_+$. 


For $V\cong X_4^0$, we have two toric fibrations $V\to\mathbb{P}^1$ 
associated to 
\[
\mathcal{S}_2\to\mathcal{S}_2/\left(\mathbb{R}x_0+\mathbb{R}y_+
+\mathbb{R}y_-+\mathbb{R}z_1+\mathbb{R}z_2
\right)
\mbox{ and }
\mathcal{S}_2\to\mathcal{S}_2/\left(\mathbb{R}x_0+\mathbb{R}x_+
+\mathbb{R}x_-+\mathbb{R}z_2\right).
\]
Let $X_1$ and $X_2$ be smooth toric weak Fano $4$-folds constructed 
by adding primitive relations 
$x_1+x_2=2x_+$ and $x_1+x_2=2y_-$, respectively. 
$X_1$ is not special as in the case where $V\cong X_3^0$. 
Thus, we obtain one new smooth toric special weak Fano $4$-fold $\mathcal{Z}_{25}:=X_2$. 
In this case, $\varphi(E)\cong F_2$. 

\medskip

Put
\[
x_0=(1,0,0),\ x_+=(1,1,0),\ x_-=(1,-1,0),\ y_+=(0,0,1),\ y_-=(0,1,-1),
\]
\[
z_1=(0,1,0),\ z_2=(-1,0,0),\ z_3=(0,-1,0).
\]
Then, we have $V\cong X_4^1$ if $\G(\Sigma_V)=\{x_0,x_+,x_-,y_+,y_-,
z_1,z_2\}$, while $V\cong X_5^1$ if $\G(\Sigma_V)=\{x_0,x_+,x_-,y_+,y_-,
z_1,z_2,z_3\}$ (see Lemma \ref{anticanonical}). 

For $V\cong X_4^1$, there exists a toric fibration $V\to\mathbb{P}^1$ 
associated to the projection
\[
\mathcal{S}_2\to\mathcal{S}_2/\left(\mathbb{R}x_0+\mathbb{R}x_+
+\mathbb{R}x_-+\mathbb{R}z_1+\mathbb{R}z_2
\right). 
\]
Thus, 
we have a smooth toric 
special weak Fano $4$-fold by adding the primitive relation $x_1+x_2=2y_+$. 
$X$ is isomorphic to a $W_3$-bundle over $F_2$ in this case. 
In fact, 
this manifold is isomorphic to $\mathcal{Z}_{12}$. 

For $V\cong X_5^1$, there exist exactly two 
toric fibrations $V\to\mathbb{P}^1$ 
associated to 
\[
\mathcal{S}_2\to\mathcal{S}_2/\left(\mathbb{R}x_0+\mathbb{R}x_+
+\mathbb{R}x_-+\mathbb{R}z_1+\mathbb{R}z_2+\mathbb{R}z_3
\right)
\mbox{ and }
\]
\[
\mathcal{S}_2\to\mathcal{S}_2/\left(\mathbb{R}y_++\mathbb{R}y_-
+\mathbb{R}z_1+\mathbb{R}z_3
\right).
\]
In the former case, $X$ is a $W_4^1$-bundle over $F_2$ 
and is isomorphic to $\mathcal{Z}_{13}$. 
In the latter case, we have a new smooth toric 
special weak Fano $4$-fold $\mathcal{Z}_{26}$ which has one more primitive relation 
$x_1+x_2=2z_2$. In this case, $\varphi(E)\cong F_1$.

\medskip

Finally, we consider the remaining case, that is, 
the case where $\varphi(E)\cong \mathbb{P}^2$ and $V$ 
has no $\varphi(E)$-fibration over $\mathbb{P}^1$. 

In this case, one can easily see that there exist two primitive relations
\[
x_1+x_2=2x\mbox{ and }y_1+y_2+y_3=x_1+z_i,
\]
where $\G(\Sigma_{\varphi})=\{y_1,y_2,y_3\}$ and $I=\{z_1,\ldots,z_l\}$ for some 
$1\le i\le l$ 
(remember that $\G(\Sigma)=\{x_1,x_2,x\}\cup\G(\Sigma_{\varphi})\cup I$). 
We remark that $y_1+y_2+y_3=x+z_i$ is impossible, because 
$\langle x,z_i\rangle\not\in\G(\Sigma)$. 
By \cite[Proposition 3.2]{bat1}, we have one more primitive relation 
\[
x+z_j=0,
\]
since $X$ is projective and 
$\mathbb{R}x$ is the maximum subspace in $S_2$.
We may assume $j=1$. 

By using Proposition \ref{keyprop}, we can prove the following. 
The proof is similar to the one of Proposition \ref{3nenagumi}. 
We therefore omit it. 

\begin{lem}\label{mizuhara}
$\{z_1,\ldots,z_l\}$ is contained in one of the following 
$2$-dimensional cones: 
$\langle z_1,y_1\rangle$, $\langle z_1,y_2\rangle$ or 
$\langle z_1,y_3\rangle$. 
\end{lem}

Without loss of generality, we may assume 
$\{z_1,\ldots,z_l\}\subset \langle z_1,y_1\rangle$ by Lemma 
\ref{mizuhara}. Permute $z_2,\ldots,z_l$ in order of proximity to $z_1$. 
Then, 
$\{z_k,z_{k+2}\}$ is a primitive collection for $1\le k\le l-1$, where 
put $z_{l+1}:=y_1$. If there exists a primive relation 
\[
z_k+z_{k+2}=2z_{k+1}\ (1\le k\le l-1),
\]
then we have a primitive crepant contraction associated to 
this primitive relation, because there are eight 
$4$-dimensional cones 
\[
\langle z_{k+1},z_k,x_1,y_2\rangle,\ \langle z_{k+1},z_k,x_1,y_3\rangle,\ 
\langle z_{k+1},z_k,x_2,y_2\rangle,\ \langle z_{k+1},z_k,x_2,y_3\rangle,
\]
\[
\langle z_{k+1},z_{k+2},x_1,y_2\rangle,\ \langle z_{k+1},z_{k+2},x_1,y_3\rangle,\ 
\langle z_{k+1},z_{k+2},x_2,y_2\rangle
\mbox{ and }\langle z_{k+1},z_{k+2},x_2,y_3\rangle\in\Sigma
\]
which contain $z_{k+1}$. We can construct a divisorial 
contraction by removing $z_{k+1}$ 
by Proposition \ref{contractible}. 
So, this primitive crepant contraction has to be special. 
By Definition \ref{specon} (3), there exists $m'\in M$ such that 
$m'(x_1)=m'(x_2)=m'(y_2)=m'(y_3)=0$. 
However, since $x_1,x_2,y_2,y_3$ spans $N_\mathbb{R}$, this is a contradiction. 
Thus, we have primitive relations $z_k+z_{k+2}=z_{k+1}$ for $1\le k\le l-1$. 
Suppose that $l\ge 3$. By adding two primitive relations 
\[
z_1+z_3=z_2\mbox{ and }z_2+z_4=z_3,
\]
we have another primitive relation $z_1+z_4=0$. This is impossible. 
So, we have $l\le 2$.

If $l=1$, then the primitive relations of $\Sigma$ are 
\[
x_1+x_2=2x,\ y_1+y_2+y_3=x_1+z_1\mbox{ and }x+z_1=0.
\]
One can easily see that $X$ is an $F_2$-bundle over 
$\mathbb{P}^2$ in this case. 
So, this is also in the case where $I\subset\mathcal{S}_1$. 
In fact, $X$ is isomorphic to $\mathcal{Z}_1$. 

Suppose that $l=2$. If $i=1$, then we have the extremal 
primitive relation 
\[
z_2+y_2+y_3=z_1+z_3+y_2+y_3=z_1+y_1+y_2+y_3=2z_1+x_1.
\]
The associated primitive crepant contraction is not special, a contradiction. 
If $i=2$, then the relation 
\[
y_2+y_3=x_1+z_2-y_1=x_1+z_1
\]
implies that either $\{y_2,y_3\}$ or $\{x_1,z_1\}$ is a primitive collection. 
However, Proposition \ref{keyprop} says that 
$\{x_1,z_1\}$ is not a primitive collection, because 
$\G(\Sigma)\cap\mathcal{S}_1=\{x,x_1,x_2,z_1\}$. 
So, we also have a contradiction 
in this case.

Thus, we have no new smooth toric special weak Fano $4$-fold in this case. 
This finishes the classification. 

\begin{thm}\label{4foldthm}
There exist exactly $26$ smooth toric special weak Fano $4$-folds 
which do not decompose into a direct product of 
lower-dimensional varieties. 
\end{thm}

\section{Deformations of smooth toric special weak Fano 4-folds}\label{deformsec}
If a smooth toric special weak Fano $4$-fold $X$ is isomorphic to a 
direct product of lower-dimensional varieties, then $X$ is 
obviously a smooth toric weakened Fano $4$-fold. 
So, in this final section, we deal with the remaining cases. Namely, 
we show that smooth toric special weak Fano $4$-folds 
$\mathcal{Z}_{i}$ ($1\le i\le 26$) classified in Section \ref{sec4fold} 
are deformation equivalent to Fano manifolds 
with a few exceptions using the deformation families 
constructed in \cite{laface}. 
We use the notation in Batyrev's classification of smooth toric Fano $4$-folds. 

As examples, we show that $\mathcal{Z}_1$ (the case where $I\subset\mathcal{S}_1$) 
and $\mathcal{Z}_{14}$ (the case where $I\subset\mathcal{S}_2$) 
are deformation equivalent to Fano manifolds. The other cases are 
almost completely similar. However, we cannot apply this technique for the case where 
\[
\#\{v\in\G(\Sigma)\,|\,m(v)>0\}\ge 2
\] 
for a special primitive crepant contraction on $X$, 
because the general fiber of the deformation family in this case is {\em not 
necessarily} a toric 
variety in general (see \cite[Section 4]{sato2}). 
Thus, we cannot determine whether 
smooth toric special weak Fano $4$-folds 
$\mathcal{Z}_{18}$, $\mathcal{Z}_{21}$, $\mathcal{Z}_{23}$, 
$\mathcal{Z}_{25}$ and $\mathcal{Z}_{26}$ are 
deformation equivalent to Fano manifolds using our current techniques. 

\begin{ex}\label{exz1}
Let $X=X_\Sigma$ be the smooth toric special weak Fano $4$-fold of type 
$\mathcal{Z}_1$. Then, $\G(\Sigma)=\{x,x_1,x_2,x_3,y_1,y_2,y_3\}$, where 
\[
x=(-1,0,0,0),\ x_1=(-1,1,0,0),\ x_2=(-1,-1,0,0),\ x_3=(1,0,0,0),
\]
\[
y_1=(0,1,1,0),\ y_2=(0,0,0,1)\mbox{ and }y_3=(0,0,-1,-1).
\]
The primitive relations of $\Sigma$ are 
\[
x_1+x_2=2x,\ x+x_3=0\mbox{ and }y_1+y_2+y_3=x_1+x_3.
\]
In this case, $x_1+x_2=2x$ is an extremal primitive relation associated to 
a primitive special crepant contraction. $m\in M$ as 
in Definition \ref{specon} is the first projection 
$(a_1,a_2,a_3,a_4)\mapsto a_1$. Put $D_0,D_1,D_2,D_3,E_1,E_2,E_3$ 
be the torus invariant prime divisors corresponding to $x,x_1,x_2,x_3,y_1,y_2,y_3$, 
respectively. 
We use the short exact sequence 
\[
0\to M\to \mathbb{Z}^{\G(\Sigma)}\stackrel{\psi}{\to} \Pic (X)\to 0
\]
to find the degree matrix of $X$, that is, 
the matrix associated to the map $\psi$. 
Its kernel is the $4\times 7$ matrix given by the above rays, 
providing us with the relations: 
\[
-D_0-D_1-D_2+D_3=0,\ D_1-D_2+E_1=0,\ E_1-E_3=0\mbox{ and }E_2-E_3=0
\]
in $\Pic(X)$. $\{D_2,D_3,E_3\}$ is a $\mathbb{Z}$-basis for $\Pic(X)$, and 
the other torus invariant prime divisors are expressed as follows:
\[
D_0=-2D_2+D_3+E_3,\ D_1=D_2-E_3\mbox{ and }E_1=E_2=E_3.
\]
With respect to this $\mathbb{Z}$-basis, we have 
\[
      [D_0] 
    =
   \left(
    \begin{array}{c}
      -2  \\
      1   \\
      1 
    \end{array}
  \right),\ 
        [D_1] 
    =
   \left(
    \begin{array}{c}
      1  \\
      0   \\
      -1 
    \end{array}
  \right),\ 
      [D_2] 
    =
   \left(
    \begin{array}{c}
      1  \\
      0   \\
      0 
    \end{array}
    \right),\ 
      [D_3] 
    =
   \left(
    \begin{array}{c}
      0  \\
      1   \\
      0 
    \end{array}
  \right),\ 
\]
\[
      [E_1] 
    =
   \left(
    \begin{array}{c}
      0  \\
      0   \\
      1 
    \end{array}
  \right),\ 
        [E_2] 
    =
   \left(
    \begin{array}{c}
      0  \\
      0   \\
      1 
    \end{array}
  \right)\mbox{ and }
      [E_3] 
    =
   \left(
    \begin{array}{c}
      0  \\
      0   \\
      1 
    \end{array}
    \right),
    \]
where $[D]$ stands for the element in $\mathbb{Z}^3\cong\Pic(X)$ 
corresponding to a torus invariant divisor $D$. 
Thus, the degree matrix of $X$ is 
\[
\left(
    \begin{array}{ccccccc}
      [D_0] & [D_1] & [D_2] & [D_3] & [E_1] & [E_2] & [E_3] 
    \end{array}
  \right)=
   \left(
    \begin{array}{ccccccc}
      -2 & 1 & 1 & 0 & 0 & 0 & 0 \\
      1 & 0 & 0 & 1 & 0 & 0 & 0  \\
      1 & -1 & 0 & 0 & 1 & 1 & 1
    \end{array}
  \right),
\]
where $[D]$ stands for the element in $\mathbb{Z}^3\cong\Pic(X)$ 
corresponding to a torus invariant divisor $D$. 
On the other hand, the combinatorial structure of $\Sigma$ 
is determined by the primitive collections. Namely, 
a subset $S\subset\G(\Sigma)$ generates 
a cone in $\Sigma$ if and only if all of the following hold:
\[
\{x_1,x_2\}\not\subset S,\ \{x,x_3\}\not\subset S
\mbox{ and }\{y_1,y_2,y_3\}\not\subset S.
\]
Then, similarly as in Section 4 in \cite{laface}, 
there exists a one-parameter deformation $\pi:\ \mathcal{X}\to\mathbb{C}$ 
whose every fiber is a toric manifold 
such that $\mathcal{X}_0:=\pi^{-1}(0)\cong X$ and the degree matrix of a general fiber 
$\mathcal{X}_t:=\pi^{-1}(t)$ ($t\neq 0$) is obtained by replacing the first column and 
the fourth column of the one of $X$ by 
\[
[D_0]+[D_2]=
  \left(
    \begin{array}{c}
      -2 \\
      1  \\
      1 
    \end{array}
  \right)+
 \left(
    \begin{array}{c}
      1 \\
      0  \\
      0 
    \end{array}
  \right)=
 \left(
    \begin{array}{c}
      -1 \\
      1  \\
      1 
    \end{array}
  \right)
\mbox{ and }
\]
\[
[D_0]+[D_1]=
  \left(
    \begin{array}{c}
      -2 \\
      1  \\
      1 
    \end{array}
  \right)+
 \left(
    \begin{array}{c}
      1 \\
      0  \\
      -1 
    \end{array}
  \right)=
 \left(
    \begin{array}{c}
      -1 \\
      1  \\
      0 
    \end{array}
  \right),
\]
respectively. 
Here, we should remark that $x_1+x_2=2x$ corresponds to 
a primitive {\em special} crepant contraction. 
Therefore, the degree matrix of $\mathcal{X}_t$ is 
\[
  \left(
    \begin{array}{ccccccc}
      \mbox{\boldmath $-1$} & 1 & 1 & \mbox{\boldmath $-1$} & 0 & 0 & 0 \\
      \mbox{\boldmath $1$} & 0 & 0 & \mbox{\boldmath $1$} & 0 & 0 & 0  \\
      \mbox{\boldmath $1$} & -1 & 0 & \mbox{\boldmath $0$} & 1 & 1 & 1
    \end{array}
  \right).
\]
By adding the second line to the first line, we obtain the matrix 
\[
    \left(
    \begin{array}{ccccccc}
      \mbox{\boldmath $0$} & 1 & 1 & 0 & 0 & 0 & 0 \\
      1 & 0 & 0 & 1 & 0 & 0 & 0  \\
      1 & -1 & 0 & 0 & 1 & 1 & 1
    \end{array}
  \right).
\]
Then, by comparing this matrix with the degree matrix of $X$, we 
put 
\[
x_2':=(1,-1,0,0).
\]
Then, one can easily construct the fan $\Sigma'$ associated to $\mathcal{X}_t$ 
by replacing $x_2\in\G(\Sigma)$ by $x_2'$. 
Namely, $\G(\Sigma')=\{x,x_1,x_2',x_3,y_1,y_2,y_3\}$, and the combinatorial 
structure of $\Sigma'$ is completely similar as $\Sigma$, that is, 
the primitive relations of $\Sigma'$ are 
\[
x_1+x_2'=0,\ x+x_3=0\mbox{ and }y_1+y_2+y_3=x_1+x_3.
\]
Thus, $\mathcal{X}_t$ is a smooth toric Fano $4$-fold of type $D_9$. 
In particular, $X$ is a smooth toric weakened Fano $4$-fold. 
\end{ex}

\begin{ex}\label{exz14} 
Let $X=X_\Sigma$ be the smooth toric special weak Fano $4$-fold of type 
$\mathcal{Z}_{14}$. Then, $\G(\Sigma)=\{x,x_1,x_2,y_1,y_2,y_3,z\}$, where 
\[
x=(-1,0,0,0),\ x_1=(-1,1,0,0),\ x_2=(-1,-1,0,0),
\]
\[
y_1=(0,0,1,0),\ y_2=(0,0,0,1),\ y_3=(0,0,-1,-1)\mbox{ and }z=(1,0,1,0).
\]
The primitive relations of $\Sigma$ are 
\[
x_1+x_2=2x,\ y_1+y_2+y_3=0\mbox{ and }x+z=y_1.
\]
Put $D_0,D_1,D_2,E_1,E_2,E_3,F$ 
be the torus invariant prime divisors corresponding to $x,x_1,$
$x_2,y_1,y_2,y_3,z$, 
respectively. 
We use the short exact sequence 
\[
0\to M\to \mathbb{Z}^{\G(\Sigma)}\stackrel{\psi}{\to} \Pic (X)\to 0
\]
to find the degree matrix of $X$, that is, 
the matrix associated to the map $\psi$. 
Its kernel is the $4\times 7$ matrix given by the above rays, 
providing us with the relations: 
\[
-D_0-D_1-D_2+F=0,\ D_1-D_2=0,\ E_1-E_3+F=0\mbox{ and }E_2-E_3=0
\]
in $\Pic(X)$. $\{D_2,E_3,F\}$ is a $\mathbb{Z}$-basis for $\Pic(X)$, and 
the other torus invariant prime divisors are expressed as follows:
\[
D_0=-2D_2+F,\ D_1=D_2,\ E_1=E_3-F\mbox{ and }E_2=E_3.
\]
With respect to this $\mathbb{Z}$-basis, we have 
\[
      [D_0] 
    =
   \left(
    \begin{array}{c}
      -2  \\
      0   \\
      1 
    \end{array}
  \right),\ 
        [D_1] 
    =
   \left(
    \begin{array}{c}
      1  \\
      0   \\
      0 
    \end{array}
  \right),\ 
      [D_2] 
    =
   \left(
    \begin{array}{c}
      1  \\
      0   \\
      0 
    \end{array}
    \right),\ 
      [E_1] 
    =
   \left(
    \begin{array}{c}
      0  \\
      1   \\
      -1 
    \end{array}
  \right),\ 
\]
\[
      [E_2] 
    =
   \left(
    \begin{array}{c}
      0  \\
      1  \\
      0 
    \end{array}
  \right),\ 
        [E_3] 
    =
   \left(
    \begin{array}{c}
      0  \\
      1   \\
      0 
    \end{array}
  \right)\mbox{ and }
      [F] 
    =
   \left(
    \begin{array}{c}
      0  \\
      0   \\
      1 
    \end{array}
    \right).
    \]
Thus, the degree matrix of $X$ is 
\[
\left(
    \begin{array}{ccccccc}
      [D_0] & [D_1] & [D_2]  & [E_1] & [E_2] & [E_3] & [F]
    \end{array}
  \right)=
   \left(
    \begin{array}{ccccccc}
      -2 & 1 & 1 & 0 & 0 & 0 & 0 \\
      0 & 0 & 0 & 1 & 1 & 1 & 0  \\
      1 & 0 & 0 & -1 & 0 & 0 & 1
    \end{array}
  \right). 
\]
There exists a one-parameter deformation $\pi:\ \mathcal{X}\to\mathbb{C}$ 
such that $\mathcal{X}_0\cong X$ and the degree matrix of a general fiber 
$\mathcal{X}_t$ ($t\neq 0$) is obtained by replacing the first column and 
the last column of the one of $X$ by 
\[
[D_0]+[D_2]=
  \left(
    \begin{array}{c}
      -2 \\
      0  \\
      1 
    \end{array}
  \right)+
 \left(
    \begin{array}{c}
      1 \\
      0  \\
      0 
    \end{array}
  \right)=
 \left(
    \begin{array}{c}
      -1 \\
      0  \\
      1 
    \end{array}
  \right)
\mbox{ and }
\]
\[
[D_0]+[D_1]=
  \left(
    \begin{array}{c}
      -2 \\
      0  \\
      1 
    \end{array}
  \right)+
 \left(
    \begin{array}{c}
      1 \\
      0  \\
      0 
    \end{array}
  \right)=
 \left(
    \begin{array}{c}
      -1 \\
      0  \\
      1 
    \end{array}
  \right),
\]
respectively. Therefore, the degree matrix of $\mathcal{X}_t$ is 
\[
  \left(
    \begin{array}{ccccccc}
      \mbox{\boldmath $-1$} & 1 & 1 & 0 & 0 & 0 & \mbox{\boldmath $-1$} \\
      \mbox{\boldmath $0$} & 0 & 0 & 1 & 1 & 1 & \mbox{\boldmath $0$}   \\
      \mbox{\boldmath $1$} & 0 & 0 & -1 & 0 & 0 & \mbox{\boldmath $1$} 
    \end{array}
  \right).
\]
By adding the third line to the first line, we obtain the matrix 
\[  
    \left(
    \begin{array}{ccccccc}
      \mbox{\boldmath $0$} & 1 & 1 & \mbox{\boldmath $-1$} & 0 & 0 & 0 \\
      0 & 0 & 0 & 1 & 1 & 1 & 0  \\
      1 & 0 & 0 & -1 & 0 & 0 & 1
    \end{array}
  \right).
\]
Then, by comparing this matrix with the degree matrix of $X$, we put 
\[
x_2':=(1,-1,1,0).
\]
Then, one can easily construct the fan $\Sigma'$ associated to $\mathcal{X}_t$ 
by replacing $x_2\in\G(\Sigma)$ by $x_2'$. 
Namely, $\G(\Sigma')=\{x,x_1,x_2',y_1,y_2,y_3,z\}$, and the combinatorial 
structure of $\Sigma'$ is completely similar as $\Sigma$, that is, 
the primitive relations of $\Sigma'$ are 
\[
x_1+x_2'=y_1,\ y_1+y_2+y_3=0\mbox{ and }x+z=y_1.
\]
Thus, $\mathcal{X}_t$ is a smooth toric Fano $4$-fold of type $D_7$. 
In particular, $X$ is a smooth toric weakened Fano $4$-fold. 
\end{ex}

Here, we summarize the central fibers $\mathcal{X}_0$ and 
the general fibers $\mathcal{X}_t$ for one-parameter 
deformation families $\mathcal{X}\to\mathbb{C}$ as in 
Examples \ref{exz1} and \ref{exz14}. 

\begin{table}[htb]
\begin{center}
\begin{tabular}{|c||c|c|c|c|c|c|c|c|c|c|c|c|c|}
\hline
special weak Fano 
& $\mathcal{Z}_1$ & $\mathcal{Z}_2$ & $\mathcal{Z}_3$ 
& $\mathcal{Z}_4$ & $\mathcal{Z}_5$ & $\mathcal{Z}_6$ 
& $\mathcal{Z}_7$ & $\mathcal{Z}_8$ & $\mathcal{Z}_9$ & $\mathcal{Z}_{10}$ 
& $\mathcal{Z}_{11}$ & $\mathcal{Z}_{12}$ & $\mathcal{Z}_{13}$\\
\hline
deformed to & $D_9$ & $H_4$ & $H_1$ & $H_{10}$ & $K_3$ & $K_1$ 
& $Q_1$ & $U_1$ & $U_8$ & $Q_2$ & $U_2$ & $\mathcal{Z}_{19}$ & 
$\mathcal{Z}_{24}$ \\ \hline

\end{tabular}
\end{center}
\end{table}

\begin{table}[htb]
\begin{center}
\begin{tabular}{|c||c|c|c|c|c|c|c|c|}
\hline
special weak Fano 
& $\mathcal{Z}_{14}$ & $\mathcal{Z}_{15}$ & $\mathcal{Z}_{16}$ 
& $\mathcal{Z}_{17}$ & $\mathcal{Z}_{19}$ & $\mathcal{Z}_{20}$ 
& $\mathcal{Z}_{22}$ & $\mathcal{Z}_{24}$ \\
\hline
deformed to & $D_7$ & $L_1$ & $L_2$ & $L_{13}$ & $Q_1$ & $Q_3$ 
& $Q_{13}$ & $U_1$  \\ \hline

\end{tabular}
\end{center}
\end{table}

By combining with the case of direct products, we obtain the following:

\begin{thm}\label{deformmain}
Let $X$ be a smooth toric special weak Fano $4$-fold. 
If $X$ is not isomorphic to $\mathcal{Z}_{18}$, $\mathcal{Z}_{21}$, $\mathcal{Z}_{23}$, $\mathcal{Z}_{25}$ or $\mathcal{Z}_{26}$, 
then $X$ is deformation equivalent to a toric Fano manifold. 
Moreover, if $X$ is not isomorphic to $\mathcal{Z}_{12}$ or $\mathcal{Z}_{13}$, 
then $X$ is a smooth toric weakened Fano $4$-fold.
\end{thm}
\begin{rem}
$\mathcal{Z}_{12}$ and $\mathcal{Z}_{13}$ become toric Fano manifolds 
after two deformations
\[
\mathcal{Z}_{12}\sim\mathcal{Z}_{19}\sim Q_1\mbox{ and }
\mathcal{Z}_{13}\sim\mathcal{Z}_{24}\sim U_1, 
\]
respectively. However, we do not know if they 
deform to Fano manifolds {\em directly}. 
\end{rem}

\end{document}